\documentclass{article}

\usepackage[main,preprint]{neurips_2026}

\usepackage[utf8]{inputenc}
\usepackage[T1]{fontenc}
\usepackage{courier} 
\usepackage{amsmath,amssymb,amsthm,mathtools,mathrsfs}
\usepackage{array}
\usepackage{booktabs}
\usepackage{enumitem}
\usepackage{graphicx}
\usepackage{tikz}
\usetikzlibrary{arrows.meta,calc,positioning,decorations.pathreplacing}
\usepackage{float}
\usepackage{microtype}
\usepackage{xcolor}
\usepackage{url}
\usepackage{hyperref}
\hypersetup{colorlinks=true,linkcolor=blue,citecolor=blue,urlcolor=blue}

\newtheorem{theorem}{Theorem}[section]
\newtheorem{lemma}[theorem]{Lemma}
\newtheorem{corollary}[theorem]{Corollary}
\newtheorem{hypothesis}{Hypothesis}[section]
\theoremstyle{definition}
\newtheorem{definition}[theorem]{Definition}
\newtheorem{remark}[theorem]{Remark}

\newcommand{\R}{\mathbb{R}}
\newcommand{\Q}{\mathbb{Q}}
\newcommand{\N}{\mathbb{N}}
\newcommand{\Z}{\mathbb{Z}}
\newcommand{\Ball}{\mathbb{B}}
\newcommand{\conv}{\operatorname{conv}}
\newcommand{\cl}{\operatorname{cl}}
\newcommand{\ri}{\operatorname{ri}}
\newcommand{\dist}{\operatorname{dist}}
\newcommand{\argmax}{\operatorname*{arg\,max}}
\newcommand{\argmin}{\operatorname*{arg\,min}}
\newcommand{\poly}{\mathrm{poly}}
\newcommand{\eps}{\varepsilon}

\newcommand{\MC}{\mathrm{MC}}
\newcommand{\Fsub}{\partial_{\mathrm F}}
\newcommand{\Csub}{\partial_{\mathrm C}}
\newcommand{\cone}{\operatorname{cone}}
\newcommand{\Forb}{\mathrm{Forb}}
\newcommand{\vertices}{\operatorname{vert}}
\newcommand{\vol}{\operatorname{vol}}
\newcommand{\relu}{\operatorname{ReLU}}
\newcommand{\size}{\mathsf{s}}

\title{Parameterized Complexity of Stationarity Testing for Piecewise-Affine Functions and Shallow CNN Losses}
\author{Yuhan Ye\\ MIT\\ \texttt{yyh03@mit.edu}}
\date{}

\begin{document}
\maketitle

\begin{abstract}
We study the parameterized complexity of testing approximate first-order stationarity at a prescribed point for continuous piecewise-affine (PA) functions, a basic task in nonsmooth optimization.
PA functions form a canonical model for nonsmooth stationarity testing and capture the local polyhedral geometry that appears in ReLU-type training losses.
Recent work of \citet{TianSo2025} shows that testing approximate stationarity notions for PA functions is computationally intractable in the worst case, and identifies fixed-dimensional tractability as an open direction.
We address this direction from the viewpoint of parameterized complexity, with the ambient dimension $d$ as the parameter.
In this paper, we give XP algorithms in fixed dimension for the tractable sides, and prove W[1]-hardness for the complementary sides.
Moreover, lower bounds under the parameterized $k$-Clique hypothesis rule out deterministic algorithms running in time $\rho(d)\size^{o(d)}$ for any computable function $\rho$, where $\size$ denotes the total binary encoding length of the instance.
As a further consequence, our results yield the corresponding parameterized complexity picture for testing local minimality of continuous PA functions.
We also extend our hardness results to a family of shallow ReLU CNN training losses, tested in the trainable weight space.
Thus, the same parameterized-complexity picture also appears for simple CNN training losses.
Notably, our lower bounds assume that $k$-Clique on $N$-vertex graphs admits no deterministic $\rho(k)N^{o(k)}$-time algorithm for any computable function $\rho$. The recent refutations of the Exponential Time Hypothesis (ETH)~\citep{OpenAI2026ETH} and the stronger fine-grained $k$-Clique hypothesis~\citep{VassilevskaWilliamsZamir2026Clique} do not contradict this assumption.
\end{abstract}

\section{Introduction}
For a continuously differentiable objective $f:\R^d\to\R$, a point is stationary precisely when its gradient vanishes.
Testing stationarity at a prescribed point is then straightforward.
For nonsmooth objectives, the situation is more subtle: several generalized subdifferentials coexist, and the resulting stationarity notions can differ substantially~\citep{Clarke1990,RockafellarWets1998,Mordukhovich2006}.
Moreover, deciding whether a prescribed point satisfies a given nonsmooth
stationarity notion can itself be computationally nontrivial.
This issue is intrinsic to ReLU-type training losses, where nondifferentiability
is an architectural feature rather than a degenerate exception~\citep{TianSo2025,YunSraJadbabaie2019}.

We study the computational complexity of testing approximate first-order stationarity at a prescribed point, which is different from the search problem of finding a stationary point. As a foundation for understanding more general piecewise differentiable objectives, we focus on continuous piecewise-affine (PA) functions. The polyhedral local geometry of PA objectives is also expressive enough to
capture local directional models arising in ReLU and max-pooling training
losses~\citep{RisterRubin2017,YunSraJadbabaie2019,TianSo2025}. The two stationarity notions considered throughout are Fr\'echet and Clarke stationarity.\footnote{Formal definitions are given in Section~\ref{subsec:stationarity}.} For continuous PA objectives, exact Fr\'echet stationarity is equivalent to local minimality. Recent work of \citet{TianSo2025} showed that approximate PA stationarity testing is already computationally intractable in the worst-case sense, and raised fixed-dimensional tractability as a natural open direction.

Against this background, we ask how the complexity changes when the ambient dimension $d$ is treated as a parameter.
Let $\size$ denote the total binary encoding length of the PA input formula, query point, and the stationarity approximation tolerance.
The most direct enumeration algorithms give running times of the form $\size^{O(d)}$ for explicit max--min input and $\size^{\Gamma(d)}$ for fixed-depth $n$-$\MC$ DC input\footnote{Formal definitions are given in Section~\ref{subsec:PArerepresent}.}, where the computable exponent $\Gamma$ may depend on the fixed depth.
In parameterized-complexity terminology, such bounds place the corresponding fixed-depth problems in XP.
The central question is whether one can substantially beat this direct enumeration dependence---in particular, obtain fixed-parameter tractable (FPT) algorithms with running time $f(d)\size^{O(1)}$---or whether a $\size^{\Theta(d)}$ running time is unavoidable. 

This fixed-dimension viewpoint is closely connected to recent work on ReLU training and verification.
For two-layer ReLU training, fixed data dimension yields natural enumeration algorithms, but this does not remove hardness~\citep{AroraEtAl2018,FroeseHertrichNiedermeier2022,FroeseHertrich2023}.
For ReLU verification, practical verification tools and complexity results have motivated a detailed study of the functions computed by ReLU networks~\citep{KatzEtAl2017Reluplex,KatzEtAl2019Marabou,FroeseGrilloSkutella2025}.
A COLT 2025 open-problem note by \citet{FroeseOpenProblem2025} asks whether several zonotope problems arising from ReLU verification are fixed-parameter tractable with respect to the input dimension.
The follow-up work of \citet{FroeseGrilloHertrichStargalla2025} proves W[1]-hardness and ETH-tight lower bounds for zonotope containment and related ReLU verification tasks.
Our results show that an analogous parameterized-complexity picture arises for stationarity testing of PA functions, and we further transfer this picture to a family of shallow ReLU convolutional neural network (CNN) training losses.

\subsection{Our contributions}
We use the following convention for the stationarity tests of PA functions $f$.
Fr\'echet-$\eps$-YES asks whether $0\in\Fsub f(w)+\eps\Ball$, and Clarke-$\eps$-NO asks whether $\dist(0,\Csub f(w))>\eps$; the other two sides are the complementary Fr\'echet-$\eps$-NO and Clarke-$\eps$-YES tests.
Our contributions are as follows.

\begin{enumerate}[label=\textnormal{(\roman*)},leftmargin=2.2em]
\item \textbf{Fixed-dimensional algorithms.}
For every $\eps\ge0$, Fr\'echet-$\eps$-YES and Clarke-$\eps$-NO are in XP.  For explicit max--min input, both tests are decidable in time $\size^{O(d)}$.  For fixed-depth explicit $n$-$\MC$ DC input, both tests are decidable in time $\size^{\Gamma(d)}$, where $\Gamma$ is a computable function that may depend on the fixed depth.

\item \textbf{Parameterized hardness and lower bounds.}
For the complementary sides, $d$-Fr\'echet-$\eps$-NO is W[1]-hard for every fixed $\eps\ge0$, and $d$-Clarke-$\eps$-YES is W[1]-hard for every fixed $\eps\in[0,1/2)$.  The reductions are parameter-preserving from $k$-Clique.  Under Hypothesis~\ref{hyp:clique}, neither hard side admits a deterministic algorithm running in time $\rho(d)\size^{o(d)}$ for any computable function $\rho$.

\item \textbf{Consequences for local minimality and CNN losses.}
Our results give the corresponding dimension-parameterized picture for testing whether a prescribed point is a local minimizer of $f$.  
We also construct a shallow ReLU CNN training loss family whose Fr\'echet nonstationarity and Clarke stationarity test are W[1]-hard with respect to the trainable parameter dimension.
Hence, even this basic stationarity and local-minimality test can be fixed-parameter intractable for shallow CNN losses.
\end{enumerate}

\begin{table}[t]
\centering
\caption{Summary of the complexity landscape.}
\label{tab:complexity-landscape}
\renewcommand{\arraystretch}{1.12}
{\setlength{\tabcolsep}{4pt}
\begin{tabular}{>{\raggedright\arraybackslash}p{0.25\linewidth} >{\raggedright\arraybackslash}p{0.32\linewidth} >{\raggedright\arraybackslash}p{0.15\linewidth} >{\raggedright\arraybackslash}p{0.2\linewidth}}
\toprule
Problem & Description & Complexity & Reference\\
\midrule
Fr\'echet-$\eps$-YES & $0\in \Fsub f(w)+\eps\Ball$ & coNP-hard & \citep{TianSo2025}\\
Clarke-$\eps$-YES & $0\in \Csub f(w)+\eps\Ball$ & NP-hard & \citep{TianSo2025}\\
$d$-Fr\'echet-$\eps$-YES & $0\in \Fsub f(w)+\eps\Ball$ & in XP & Theorem~\ref{thm:fixed-d-xp}\\
$d$-Clarke-$\eps$-NO & $\dist(0,\Csub f(w))>\eps$ & in XP & Theorem~\ref{thm:fixed-d-xp}\\
$d$-Fr\'echet-$\eps$-NO & $\dist(0,\Fsub f(w))>\eps$ & W[1]-hard & Theorem~\ref{thm:w1-hard}\\
$d$-Clarke-$\eps$-YES & $0\in \Csub f(w)+\eps\Ball$ & W[1]-hard & Theorem~\ref{thm:w1-hard}\\
\addlinespace[2pt]
\multicolumn{4}{@{}>{\raggedright\arraybackslash}p{\linewidth}@{}}{\centering\emph{CNN training losses $L$: shallow ReLU CNN with max-pooling; parameter $p=\dim(\theta)$.}}\\
\addlinespace[1pt]
\hspace{0.75em}Fr\'echet-$\eps$-NO & $\dist(0,\Fsub L(\theta_0))>\eps$ & W[1]-hard in $p$ & Theorem~\ref{thm:cnn-hardness}\\
\hspace{0.75em}Clarke-$\eps$-YES & $0\in \Csub L(\theta_0)+\eps\Ball$ & W[1]-hard in $p$ & Theorem~\ref{thm:cnn-hardness}\\
\bottomrule
\end{tabular}}
\end{table}

\subsection{Related work}
\paragraph{Stationarity testing for PA functions and ReLU network losses.}
The complexity of testing whether a prescribed point satisfies a local solution concept has a long history.
Classical results establish hardness for local optimality in quadratic and nonlinear optimization, and later work studies local-minimum testing and related local solution concepts for polynomial and PA objectives~\citep{MurtyKabadi1987,Nesterov2013,AhmadiZhang2022ComplexityAspects,AhmadiZhang2022LocalMin}.
\citet{TianSo2025} developed a PA framework for stationarity testing that isolates the representation-level polyhedral core of the problem, proves coNP-hardness for approximate Fr\'echet stationarity and NP-hardness for approximate Clarke stationarity, and develops near-stationarity algorithms and regularity conditions.
An earlier version appeared at the NeurIPS 2023 OPT workshop~\citep{TianSo2023OPT}.
Their results also include applications to structured nonsmooth neural-network losses, including shallow ReLU-type networks and CNNs.
Local optimality and saddle phenomena for shallow ReLU losses at nondifferentiable points have also been studied through explicit polyhedral analysis~\citep{YunSraJadbabaie2019}.

\paragraph{Approximate stationarity in nonsmooth optimization.}
Approximate stationarity in nonsmooth nonconvex optimization has been studied from several algorithmic and oracle-complexity perspectives.
Representative directions include subdifferential-based stationarity notions~\citep{LiSoMa2020}, algorithms for finding approximate or Goldstein-type stationary points~\citep{ZhangEtAl2020,TianZhouSo2022,JordanEtAl2023}, stochastic and online algorithms for nonsmooth DC optimization~\citep{YeCuiWang2025OnlineAdaptiveSampling}, descent-oriented subgradient methods for nonsmooth optimization~\citep{LiCui2025SubgradientRegularization}, and lower bounds or dimension-free impossibility results~\citep{KornowskiShamir2021,TianSo2024}.
These works mostly concern oracle or algorithmic models for finding approximate stationary points.
In contrast, our task is to decide a stationarity test with dimension as the hardness parameter.

\paragraph{Parameterized complexity of fixed-dimensional ReLU training and verification.}
Fixed-dimensional ReLU training and verification provide a closely related parameterized-complexity backdrop.
For two-layer ReLU training, fixed data dimension yields natural enumeration algorithms, but this does not eliminate hardness~\citep{AroraEtAl2018,FroeseHertrichNiedermeier2022,FroeseHertrich2023}.
A complementary line studies the polyhedral partition structure of ReLU
networks through their linear regions: classical work gives expressivity bounds
for such regions, while recent work studies the complexity of counting them
~\citep{MontufarEtAl2014,StargallaHertrichReichman2025}.
For ReLU verification, algorithmic systems such as Reluplex and Marabou helped establish neural-network verification as a central computational task~\citep{KatzEtAl2017Reluplex,KatzEtAl2019Marabou}, while recent complexity work connects ReLU positivity, surjectivity, and verification to zonotope containment~\citep{FroeseGrilloSkutella2025,KulmburgAlthoff2021}.
The COLT 2025 open-problem of \citet{FroeseOpenProblem2025} formulated the fixed-parameter tractability of several zonotope problems as a central question.
Subsequent work proved W[1]-hardness and ETH-tight lower bounds for zonotope containment and related ReLU verification tasks~\citep{FroeseGrilloHertrichStargalla2025}.
These works concern global training or verification problems in fixed input dimension, whereas our focus is local stationarity testing at a prescribed point.

\section{Preliminaries}\label{sec:prelim}

\noindent
All numerical input data are rationally encoded.
We write $[m]:=\{1,\dots,m\}$.
All norms are Euclidean, and $\Ball:=\{u\in\R^d:\|u\|_2\le1\}$ denotes the closed unit ball in the ambient space.
For a set $S\subseteq\R^d$, $\conv(S)$, $\cl(S)$, $\ri(S)$, and $\cone(S)$ denote its convex hull, closure, relative interior, and conic hull, respectively.
For a polytope $P$, $\vertices(P)$ denotes its vertex set.
For a nonempty compact convex set $K\subseteq\R^d$, its support function is
$\sigma_K(u):=\max_{z\in K}\langle z,u\rangle$. For a polytope $P\subseteq\R^d$ and a point $v\in P$, we write $N_P(v):=\{u\in\R^d:\langle v,u\rangle=\sigma_P(u)\}$ for the normal cone of $P$ at $v$; $\ri N_P(v)$ denotes the relative interior of this cone. We use the convention $\dist(0,\emptyset)=+\infty$.

For upper bounds, $\size$ denotes the total binary encoding length of the relevant PA input formula together with the query point and, when it is part of the input, the stationarity approximation tolerance \(\eps\ge 0\).
When the dimension dependence is displayed explicitly, $\size^{O(d)}$ means $\size^{c d+c_0}$ for absolute constants $c,c_0$ independent of  other parameters of the instance.
For every fixed $d$, the exponent $c d+c_0$ is therefore a constant depending only on $d$, so such a bound is of the form $\size^{O_d(1)}$.
For hardness reductions, $N$ denotes the number of vertices in the input graph and also the total encoding size when no confusion can arise.

\subsection{Piecewise-affine representations}
\label{subsec:PArerepresent}
A continuous function $f:\R^d\to\R$ is piecewise-affine (PA) if $\R^d$ admits a finite polyhedral partition such that $f$ is affine on every cell.
Continuous PA functions admit several equivalent finite descriptions, including max--min and DC representations; see, for example, \citet{Ovchinnikov2002} and \citet{Scholtes2012}.
For complexity-theoretic purposes, the representation and its encoding length are part of the input model.
We use two compact input models.
The first exposes max--min structure explicitly.

\begin{definition}[Explicit max--min representation]\label{def:maxmin}
A PA function $f:\R^d\to\R$ is given in explicit max--min form if there are rational affine functions $\alpha_j(w)=\langle x_j,w\rangle+a_j$ for $j\in[k]$, with $(x_j,a_j)\in\Q^{d+1}$, and nonempty index sets $M_1,\dots,M_\ell\subseteq[k]$ such that
\[
 f(w)=\max_{1\le i\le \ell}\ \min_{j\in M_i}\alpha_j(w).
\]
\end{definition}

For example, $\max\{\min(\alpha_1,\alpha_2),\alpha_3\}$ is an explicit max--min formula with $M_1=\{1,2\}$ and $M_2=\{3\}$.
The second model is a DC input model. We follow the nested max/sum syntax of \citet{TianSo2025} for the convex PA components.

\begin{definition}[$n$-$\MC$ representation]\label{def:mc}
A convex PA function $h:\R^d\to\R$ is given in explicit $n$-$\MC$ form if there are integers $n\in\N$ and positive integers $(I_p,J_p)$, $p\in[n]$, together with rational affine leaf data
\[
\bigl(x_{i_1,j_1,\dots,i_n,j_n},a_{i_1,j_1,\dots,i_n,j_n}\bigr)\in\Q^{d+1}
\]
indexed by $(i_1,j_1,\dots,i_n,j_n)\in\prod_{p=1}^n([I_p]\times[J_p])$ such that
\[
 h(w)=\sum_{1\le j_n\le J_n}\max_{1\le i_n\le I_n}\cdots
 \sum_{1\le j_1\le J_1}\max_{1\le i_1\le I_1}
 \bigl(\langle w,x_{i_1,j_1,\dots,i_n,j_n}\rangle+a_{i_1,j_1,\dots,i_n,j_n}\bigr).
\]
For example, $h(w)=\sum_{j=1}^m\max\{\alpha_j(w),0\}$ is a simple $1$-$\MC$ formula.
A DC representation of a PA function is an expression $f=h-g$ in which $h$ and $g$ are convex PA functions given in $n_h$-$\MC$ and $n_g$-$\MC$ form, respectively.  The finite expression tree is given as part of the input: the ranges \((I_p,J_p)\), the alternating max/sum structure, and all affine leaf data are listed.
The depths \(n_h,n_g\) contribute to the encoding length but are not parameters; in fixed-depth families they are bounded by an input-independent constant.
\end{definition}

\paragraph{V-representations.}
We also use the following polyhedral representation convention.
A V-representation of a polytope $P\subseteq\R^d$ is a finite rational set
$V\subseteq\Q^d$ such that $P=\conv(V)$; the set $V$ need not be irredundant.
In the algorithmic arguments below, whenever we write $\vertices(P)$, we mean the
reduced vertex list of $P$. Given a finite rational V-representation $V$, this
reduction can be performed by testing, for each $v\in V$, whether
$v\in\conv(V\setminus\{v\})$, which is a rational linear feasibility problem.

\subsection{Parameterized complexity, ETH, and Clique hypotheses}
We use standard parameterized-complexity terminology~\citep{CyganEtAl2015}.
A parameterized problem with parameter $k$ is fixed-parameter tractable (FPT) if it is solvable in time $f(k)|x|^{O(1)}$ for some computable function $f$.
It is in XP if it is solvable in time $|x|^{\Gamma(k)}$ for some computable function $\Gamma$.
The class W[1] can be viewed as a parameterized analogue of NP for problems whose witnesses have size controlled by the parameter; the canonical W[1]-complete problem is $k$-Clique parameterized by the target clique size $k$.
A parameterized reduction maps an instance $(x,k)$ to an instance $(x',k')$ in time $f(k)|x|^{O(1)}$, preserves yes/no answers, and ensures that $k'$ is bounded by a computable function of $k$.
Thus, an FPT algorithm for any W[1]-hard problem would imply FPT algorithms for all problems in W[1], in particular for $k$-Clique; this is the standard reason such hardness is taken as evidence against fixed-parameter tractability.

For lower bounds, we use the following assumption.
\begin{hypothesis}[Parameterized $k$-Clique]\label{hyp:clique}
There is no deterministic algorithm that, given an $N$-vertex graph and $k$ as input, detects a $k$-clique in time $\rho(k)N^{o(k)}$ for any computable function $\rho$.
\end{hypothesis}
The Exponential Time Hypothesis (ETH) implies Hypothesis~\ref{hyp:clique}~\citep{ChenHuangKanjXia2006}. However, the SAT algorithm that refutes ETH~\citep{OpenAI2026ETH} does not give a $\rho(k)N^{o(k)}$-time algorithm for $k$-Clique.

Moreover, the stronger fine-grained $k$-Clique hypothesis~\citep{AbboudBackursWilliams2018Parser} asserted that the exponent $\omega k/3$ was optimal for fixed $k$ divisible by three, where $\omega$ is the matrix-multiplication exponent.
\citet[Theorem~1.1]{VassilevskaWilliamsZamir2026Clique} refute this hypothesis with an $O(N^{(\omega-\delta)k/3})$-time algorithm for all sufficiently large fixed $k$ and an absolute constant $\delta>0$.
The exponent remains linear in $k$. Whether a sublinear exponent in $k$ is possible remains open~\citep[Section~5]{VassilevskaWilliamsZamir2026Clique}.
In the reductions we construct in this paper, an $N$-vertex $k$-Clique instance is transformed into a stationarity-testing instance with dimension $d=\Theta(k)$ and encoding size $\size\le N^{O(1)}$.
Hence, a target algorithm running in time $\rho(d)\size^{o(d)}$ would contradict Hypothesis~\ref{hyp:clique}.
The CNN lower bounds use the same argument with trainable dimension $p=\Theta(k)$ and network encoding size $M\le N^{O(1)}$.

\subsection{Stationarity notions}
\label{subsec:stationarity}
We write $\partial_{\mathrm C}$ and $\partial_{\mathrm F}$ for the Clarke and Fr\'echet subdifferentials, respectively.  The general definitions follow the standard nonsmooth-analysis viewpoint; see, for example, \citep{Clarke1990,RockafellarWets1998} and the monograph of \citet{CuiPang2021}. For a locally Lipschitz function $f:\R^d\to\R$, the Clarke directional derivative at $x$ in direction $u$ is
\[
 f^\circ(x;u):=\limsup_{\substack{x'\to x\\ t\downarrow0}}
 \frac{f(x'+tu)-f(x')}{t}.
\]
The Clarke subdifferential is the compact convex set
\[
 \Csub f(x):=\{g\in\R^d:\langle g,u\rangle\le f^\circ(x;u)\ \forall u\in\R^d\},
\]
equivalently characterized by $f^\circ(x;u)=\max_{g\in\Csub f(x)}\langle g,u\rangle$.
The Fr\'echet subdifferential is the set of first-order local lower-support slopes,
\[
 \Fsub f(x):=\left\{g\in\R^d:
 \liminf_{y\to x,\ y\ne x}
 \frac{f(y)-f(x)-\langle g,y-x\rangle}{\|y-x\|}\ge0\right\}.
\]
For the PA functions considered in this paper, the one-sided directional derivative
$f'(x;u):=\lim_{t\downarrow0}(f(x+tu)-f(x))/t$ exists in every direction, and the preceding definition is equivalent to
\[
 \Fsub f(x)=\{g\in\R^d:\langle g,u\rangle\le f'(x;u)\ \forall u\in\R^d\}.
\]
Intuitively, Fr\'echet subgradients are first-order local lower supports, whereas for PA functions Clarke subgradients are convex combinations of gradients on nearby full-dimensional affine regions. For locally Lipschitz functions one always has $\Fsub f(x)\subseteq\Csub f(x)$. 

A locally Lipschitz function is \emph{Clarke regular} at $x$ if $f'(x;u)$ exists and equals $f^\circ(x;u)$ for every direction $u$.  At regular points, the Fr\'echet and Clarke subdifferentials coincide in the directional description above.  This includes $C^1$ functions and finite (weakly) convex functions; in particular, every convex PA component used below is regular, so its Fr\'echet, Clarke, and usual convex subdifferentials agree. If $h$ is convex and PA, then $h'(x;u)=\sigma_{\Csub h(x)}(u)$. 

\paragraph{Stationarity-testing polarities.}
A point $x$ is Fr\'echet stationary if $0\in\Fsub f(x)$, and Clarke stationary if $0\in\Csub f(x)$. Appendix~\ref{app:stationarity-guide} gives a geometric guide to these notions. The $\eps$-stationary versions replace membership by distance at most $\eps$.
Fix a PA input formula, a query point $w$, and a stationarity approximation tolerance $\eps\ge0$.  For $\star\in\{\mathrm F,\mathrm C\}$, the $\star$-$\eps$-YES predicate is $\dist(0,\partial_{\star}f(w))\le\eps$, equivalently $0\in\partial_{\star}f(w)+\eps\Ball$, while the $\star$-$\eps$-NO predicate is $\dist(0,\partial_{\star}f(w))>\eps$.

\subsection{Local PA calculus}\label{subsec:local-pa}
The following local PA calculus lemmas will be used repeatedly; proofs are deferred to Appendix~\ref{app:prelim-proofs}.

\begin{lemma}[Exact local model for PA functions]\label{lem:pa-local-model}
Let $f:\R^d\to\R$ be continuous and PA, and fix $w\in\R^d$.
There exist $r>0$ and a positively homogeneous PA function $\phi_w:\R^d\to\R$ such that
\[
 f(w+u)=f(w)+\phi_w(u)\qquad\text{whenever }\|u\|\le r.
\]
Consequently, $f'(w;u)=\phi_w(u)$, $\Fsub f(w)=\Fsub\phi_w(0)$, and $\Csub f(w)=\Csub\phi_w(0)$.
Moreover, $0\in\Fsub f(w)$ if and only if $w$ is a local minimizer of $f$.
\end{lemma}

For explicit max--min input, the local model can be written directly from the active indices.
Let $f$ be given as in Definition~\ref{def:maxmin} and fix $w\in\R^d$.
Set $m_i(w):=\min_{j\in M_i}\alpha_j(w)$, $I(w):=\argmax_{i\in[\ell]}m_i(w)$, and $J_i(w):=\argmin_{j\in M_i}\alpha_j(w)$ for $i\in I(w)$.
Then the exact local model from Lemma~\ref{lem:pa-local-model} is
\begin{equation}\label{eq:local-model}
 \phi_w(u)=f'(w;u)=\max_{i\in I(w)}\min_{j\in J_i(w)}\langle x_j,u\rangle.
\end{equation}

Lemma~\ref{lem:frechet-translation} turns the Fr\'echet test into the translate-containment condition $s+Y\subseteq X$.

\begin{lemma}[Characterization of Fr\'echet stationarity]\label{lem:frechet-translation}
Let $h,g:\R^d\to\R$ be convex PA functions, set $f=h-g$, fix $w\in\R^d$, and put $X:=\Csub h(w)$ and $Y:=\Csub g(w)$ for the usual convex subdifferential polytopes.  Then $\Fsub(h-g)(w)=\{s\in\R^d:\ s+Y\subseteq X\}$, and
\[
 0\in\Fsub(h-g)(w)+\eps\Ball
 \iff \exists s\in\eps\Ball\text{ such that }s+Y\subseteq X\quad(\eps\ge0).
\]
In particular, $0\in\Fsub(h-g)(w)$ if and only if $Y\subseteq X$.
\end{lemma}

\begin{lemma}[Characterization of Clarke stationarity]\label{lem:dc-clarke-slopes}
Let $X,Y\subseteq\R^d$ be nonempty compact polytopes, set $\phi=\sigma_X-\sigma_Y$, and write $N_P(v):=\{u:\langle v,u\rangle=\sigma_P(u)\}$ for the normal cone of a polytope $P$ at a vertex $v$.  A pair $x\in\vertices(X)$, $y\in\vertices(Y)$ is simultaneously exposed if $\ri N_X(x)\cap\ri N_Y(y)\ne\emptyset$.  If $S_e(X,Y)$ denotes the set of differences $x-y$ over all such pairs, then $\Csub\phi(0)=\conv(S_e(X,Y))$.
\end{lemma}

The local objects above are the geometric interface used throughout the paper.  For max--min input, the active formula at $w$ gives the exact homogeneous model $\phi_w$.  For explicit $n$-$\MC$ DC input $f=h-g$, the natural objects are instead the convex subdifferential polytopes $X=\Csub h(w)$ and $Y=\Csub g(w)$. The fixed-dimensional algorithms either enumerate arrangement cones of the max--min local model or manipulate V-representations of the two convex subdifferential polytopes; the hardness reductions build these local objects so that containment or convex balancing encodes $k$-Clique.

\begin{remark}[Why the input representation matters]\label{rem:representation}
The explicit max--min and explicit $n$-$\MC$ DC models expose different local data, and neither is subsumed by the flat single-max DC baseline in Appendix~\ref{app:single-max-dc} through a polynomial-size transformation in general.

That appendix shows that, if both convex components are given in the flat form $h=\max_{a\in A}\ell_a$ and $g=\max_{b\in B}m_b$ with all affine pieces explicitly listed, then the corresponding stationarity tests can be solved in polynomial time.  However, many compact PA constructions do not fit this flat input model without an exponential expansion.  For example, the hard DC family of \citet[Problems~3.10--3.11]{TianSo2025} uses, for $y_1,\dots,y_m\in\{-1,0,1\}^n$ and $r\in\N$, the simple PA function
\[
    f_F(\xi)=h_F(\xi)-g_F(\xi)
    =r\|\xi\|_\infty-
      \max\left\{r\|\xi\|_\infty,\sum_{q=1}^m |\langle \xi,y_q\rangle|\right\}.
\]
The term $\sum_{q=1}^m|\langle \xi,y_q\rangle|$ is compact in the $n$-$\MC$ syntax, since it is a sum of two-term maxima.  In contrast, its flat single-max expansion is
$\sum_{q=1}^m |\langle \xi,y_q\rangle|=\max_{\sigma\in\{\pm1\}^m}\left\langle \xi,\sum_{q=1}^m \sigma_q y_q\right\rangle,$
which may have exponentially many affine pieces.  Thus, the single-max DC case is a useful polynomial baseline for a much flatter representation, but it does not subsume the compact $n$-$\MC$ DC instances or the explicit max--min instances studied here.  This is why we analyze the stationarity-testing complexity of both input models directly.
\end{remark}


\section{Fixed-dimension algorithms}\label{sec:fixed-d}

We first show the enumeration upper bounds when the ambient dimension is fixed. This gives the XP baseline. Throughout, $\size$ denotes the total binary encoding length of the relevant input representation together with the query point and the stationarity approximation tolerance \(\eps\ge 0\). 

\begin{theorem}\label{thm:fixed-d-xp}
Let $d$ be the ambient dimension. Consider the following input models:
\begin{enumerate}[label=\textnormal{(\alph*)},leftmargin=2.2em]
\item explicit max--min representation;
\item fixed-depth explicit $n$-$\MC$ DC representation $f=h-g$.
\end{enumerate}
For every $\eps\ge0$, the following problems are decidable by deterministic enumeration:
\begin{enumerate}[label=\textnormal{(\roman*)},leftmargin=2.2em]
\item deciding whether $\dist(0,\Fsub f(w))\le\eps$ (Fr\'echet-$\eps$-YES);
\item deciding whether $w$ is a local minimizer of $f$;\footnote{By Lemma~\ref{lem:pa-local-model}, part~\textnormal{(ii)} follows from part~\textnormal{(i)} with $\eps=0$.}
\item deciding whether $\dist(0,\Csub f(w))>\eps$ (Clarke-$\eps$-NO).
\end{enumerate}
The running time is $\size^{O(d)}$ under model~\textnormal{(a)} and $\size^{\Gamma(d)}$ under model~\textnormal{(b)}, where the computable function $\Gamma$ may depend on the fixed depth. Hence the listed tests are in XP. Complementing the output gives the same upper bounds for Fr\'echet-$\eps$-NO, the non-local-minimality test, and Clarke-$\eps$-YES.
\end{theorem}

\begin{proof}[Proof sketch]

For fixed-depth explicit $n$-$\MC$ DC input, process the canonical alternating max/sum expression trees bottom-up and compute reduced vertex lists for the convex subdifferential polytopes $X=\Csub h(w)$, $Y=\Csub g(w)$.
More explicitly, affine leaves give singleton polytopes, sum nodes give Minkowski sums, and max nodes give convex hulls of active children. For any fixed depth, both the recursive vertex output size and the enumeration time are bounded by $\size^{\Gamma(d)}$ for some computable depth-dependent exponent $\Gamma$. This gives an XP algorithm for vertex enumeration (Lemma~\ref{lem:mc-v-growth}).

Lemma~\ref{lem:frechet-translation} reduces the Fr\'echet test to finding
\(\xi\in\R^d\) with \(\|\xi\|\le\eps\) and \(\xi+Y\subseteq X\).
Write \(\vertices(X)=\{x_1,\dots,x_N\}\) and
\(\vertices(Y)=\{y_1,\dots,y_m\}\).  Since \(X=\conv\{x_1,\dots,x_N\}\)
and \(Y=\conv\{y_1,\dots,y_m\}\), the containment \(\xi+Y\subseteq X\)
is equivalent to \(\xi+y_i\in X\) for all \(i\in[m]\).  Equivalently, for
each \(i\in[m]\) there are coefficients \(\lambda_{i,r}\ge0\) such that
\(\xi+y_i=\sum_{r=1}^N\lambda_{i,r}x_r\) and
\(\sum_{r=1}^N\lambda_{i,r}=1\).  Thus Fr\'echet-\(\eps\)-YES is decided
by minimizing \(\|\xi\|_2^2\) over this rational polyhedron in the variables
\((\xi,\lambda)\) and comparing the optimum with \(\eps^2\).

For Clarke-\(\eps\)-NO, the local PA model is
\(\phi(u)=\sigma_X(u)-\sigma_Y(u)\).  By
Lemma~\ref{lem:dc-clarke-slopes}, the essentially active slopes are exactly
the differences \(x-y\) over simultaneously exposed vertex pairs. After
enumerating these slopes into \(S_e\), Clarke-\(\eps\)-NO is decided by
minimizing \(\|\sum_{a\in S_e}\lambda_a a\|_2^2\) over the simplex
\(\lambda_a\ge0,\ \sum_{a\in S_e}\lambda_a=1\), and comparing the optimum
with \(\eps^2\).

For explicit max--min input, first compute the homogeneous local model in~\eqref{eq:local-model}.
Let $A$ be the set of distinct active slopes appearing in that model.  Form the hyperplanes
$\langle a-a',u\rangle=0$ for distinct $a,a'\in A$.  These hyperplanes divide the direction space into open polyhedral cones.  We enumerate the full-dimensional ones  in time $\size^{O(d)}$~\citep{EdelsbrunnerORourkeSeidel1986}.  On each such cone $C$, all comparisons among the numbers $\langle a,u\rangle$ are fixed, so every inner minimum and the outer maximum in~\eqref{eq:local-model} select fixed slopes and $\phi_w(u)=\langle a_C,u\rangle$ on $C$.
Let $K_C=\overline C$ and write it as $K_C=\{u:B_Cu\ge0\}$.  These closures cover $\R^d$, and continuity gives $\phi_w(u)=\langle a_C,u\rangle$ on each $K_C$. Thus a vector $\xi$ belongs to $\Fsub\phi_w(0)$ exactly when, for every enumerated cone $C$, the inequality
$\langle \xi-a_C,u\rangle\le0$ holds for all $u\in K_C$.  This is equivalent to the existence of multipliers $\lambda_C\ge0$ satisfying $\xi=a_C-B_C^\top\lambda_C$.  Minimizing $\|\xi\|_2^2$ over the resulting rational polyhedron and comparing the optimum with $\eps^2$ gives a rational convex QP of size $\size^{O(d)}$.
For Clarke-$\eps$-NO, the same cone enumeration gives the essentially active slope set $S_e=\{a_C:C\text{ is a full-dimensional cone}\}$.  Since $\Csub f(w)=\conv(S_e)$, Clarke-$\eps$-NO is decided by minimizing the QP
$\min\{\|\sum_{a\in S_e}\lambda_a a\|_2^2:\lambda_a\ge0,\sum_{a\in S_e}\lambda_a=1\}$ and comparing the optimum with $\eps^2$.  All subproblems above have size $\size^{O(d)}$ under max--min input, or $\size^{\Gamma(d)}$ under fixed-depth $n$-$\MC$ DC input, and are polynomial-time solvable in that enumerated size.  Hence they give XP algorithms. Full proofs are deferred to Appendix~\ref{app:fixed-d-proofs}.
\end{proof}

\section{Parameterized hardness with respect to the input dimension}\label{sec:hardness}

We next show the parameterized W[1]-hardness and conditional running-time lower bounds for our testing problems.

\begin{theorem}\label{thm:w1-hard}
Let $d$ denote the input dimension. For each of the following input models
\begin{enumerate}[label=\textnormal{(\alph*)},leftmargin=2.2em]
\item explicit max--min representation;
\item explicit $n$-$\MC$ DC representation of constant depth (indeed depth two),
\end{enumerate}
the following decision problems are W[1]-hard with respect to $d$:
\begin{enumerate}[label=\textnormal{(\roman*)},leftmargin=2.2em]
\item for every fixed $\eps\ge0$, deciding whether $\dist(0,\Fsub f(w))>\eps$ ($d$-Fr\'echet-$\eps$-NO);
\item deciding whether $w$ is not a local minimizer of $f$;\footnote{By Lemma~\ref{lem:pa-local-model}, part~\textnormal{(ii)} follows from part~\textnormal{(i)} with $\eps=0$.}
\item for every fixed $\eps\in[0,1/2)$, deciding whether $\dist(0,\Csub f(w))\le\eps$ ($d$-Clarke-$\eps$-YES).
\end{enumerate}
Consequently, the corresponding Fr\'echet-$\eps$-YES and Clarke-$\eps$-NO polarities are co-W[1]-hard.  Finally, under Hypothesis~\ref{hyp:clique}, none of the hard problems above admits a deterministic algorithm of running time $\rho(d)\size^{o(d)}$ for any computable function $\rho$.
Combined with the $\size^{O(d)}$ max--min enumeration bound in Theorem~\ref{thm:fixed-d-xp}, this rules out improving the exponent of the natural max--min enumeration approach to $o(d)$ on these W[1]-hard tests.\footnote{Swapping the YES and NO outputs of a deterministic algorithm preserves its running time. Hence the lower bound in Theorem~\ref{thm:w1-hard} applies to both polarities of each stationarity test for every $\eps$ in its stated range, and the complementary polarities are co-W[1]-hard.}
\end{theorem}

\begin{proof}[Proof sketch]
We reduce from $k$-Clique.
Let $V(G)=[N]$, set $p_v=(v,v^2)\in\Z^2$, and put $Q=\conv\{p_v:v\in[N]\}\subseteq\R^2$.
Each $p_v$ is exposed by the direction $(2v,-1)$, since $\langle p_v-p_a,(2v,-1)\rangle=(a-v)^2>0$ for $a\ne v$.
For ordered pairs, write $(u,v)\in\Forb$ when $u=v$ or $\{u,v\}\notin E(G)$; a $k$-tuple is a clique precisely when none of its pairs is forbidden.
In block coordinates $z=(z_1,\dots,z_k)\in(\R^2)^k$, define $g(z)=\sum_{i=1}^k \sigma_Q(z_i)$, $Y:=\Csub g(0)=Q^k$.
For $1\le i<j\le k$ and $(u,v)\in\Forb$, define
\[
 \psi_{i,j,u,v}(z):=\langle p_u,z_i\rangle+\langle p_v,z_j\rangle+
 \sum_{r\in[k]\setminus\{i,j\}}\sigma_Q(z_r),
 \qquad
 h(z):=\max_{i<j,\ (u,v)\in\Forb}\psi_{i,j,u,v}(z),
\]
and let $X:=\Csub h(0)$.
Each $\psi_{i,j,u,v}$ is the support function of a product face of $Y$ with the $i$-th block fixed to $p_u$ and the $j$-th block fixed to $p_v$.
Thus $X\subseteq Y$, and a vertex $(p_{v_1},\dots,p_{v_k})$ of $Y$ belongs to $X$ exactly when some pair $(v_i,v_j)$ is forbidden.
It follows that $X=Y$ if and only if $G$ has no $k$-clique.
If $G$ has a $k$-clique, then $X$ is the convex hull of a strict subset of the vertices of the full-dimensional polytope $Y$; hence $\vol(X)<\vol(Y)$, and no translate of $Y$ can be contained in $X$.

For the Fr\'echet side, set $f_F:=h-g$.
Lemma~\ref{lem:frechet-translation} gives $\Fsub f_F(0)=\{s:s+Y\subseteq X\}$.
Therefore no-clique instances satisfy $\Fsub f_F(0)=\{0\}$, while clique instances satisfy $\Fsub f_F(0)=\emptyset$.
This proves W[1]-hardness of $d$-Fr\'echet-$\eps$-NO for every fixed $\eps\ge0$.
Since the gadget is continuous PA, Lemma~\ref{lem:pa-local-model} also turns the same dichotomy into W[1]-hardness of deciding non-local-minimality at the origin.
The explicit max--min form for $f_F(z)$ follows from $\langle p_u,z_i\rangle-\sigma_Q(z_i)=\min_{a\in[N]}\langle p_u-p_a,z_i\rangle$, giving
\[
 f_F(z)=\max_{i<j,\ (u,v)\in\Forb}\ \min_{a,b\in[N]}
 \bigl(\langle p_u-p_a,z_i\rangle+\langle p_v-p_b,z_j\rangle\bigr).
\]
The same construction has an explicit $n$-$\MC$ DC representation of constant depth, indeed depth two.

For the Clarke side, add one coordinate and define $f_C(z,t):=t/2+\max\{f_F(z),-|t|/2\}$.
If the graph has no $k$-clique, then $h=g$ and $f_F\equiv0$, so $f_C(z,t)=t/2$ and $\dist(0,\Csub f_C(0,0))=1/2$.
If the graph has a $k$-clique $(v_1,\dots,v_k)$, the exposing directions $r_i=(2v_i,-1)$ can be scaled so that $f_F(z)<-1$.
By continuity, there is a neighborhood of $(z,1)$ on which $t>0$ and the branch $-|t|/2$ is strictly selected, so $f_C=0$ there. Since $f_C$ is positively homogeneous, it also has gradient zero at points arbitrarily close to the origin. Hence $0\in\Csub f_C(0,0)$.
This proves W[1]-hardness of $d$-Clarke-$\eps$-YES for every fixed $\eps\in[0,1/2)$.
All constructions and representation expansions run in time $f(k)|(G,k)|^{O(1)}$ and produce dimension $d=2k$ for Fr\'echet and $d=2k+1$ for Clarke; hence they are valid parameterized reductions.  With polynomial encoding growth, the running-time lower bounds follow under Hypothesis~\ref{hyp:clique}.
Full proofs are given in Appendix~\ref{app:hardness-proofs}.
\end{proof}

\section{Extension to stationarity testing for ReLU CNN training losses}\label{sec:cnn}
We now connect our PA hardness results to neural-network training by showing that the same parameterized-complexity picture holds for stationarity testing of a concrete family of shallow CNN training losses.
Following the CNN embedding construction of \citet{TianSo2025}, we consider a shallow ReLU/max-pooling CNN with one fixed input tensor and a single trainable no-bias \(1\times1\) convolutional filter; all linear readouts and pooling operations are fixed and rationally encoded. 

Once the architecture and the single training example are fixed, the loss is a scalar PA function of the filter parameter \(\theta\in\R^p\).
We write \(M\) for the encoding size of the fixed tensor, fixed rational weights, pooling structure, designated parameter point, and the stationarity approximation tolerance \(\eps\ge 0\) when it is part of the input.
The formulas below define the two losses used in the reduction.\footnote{Appendix~\ref{app:cnn-layer-realization} describes the loss family and gives its exact CNN realizations.}

Let \(G=([N],E)\) be a \(k\)-Clique instance, set
\(p_v:=(v,v^2)\in\Z^2\) for \(v\in[N]\), and write
\(\mathcal N_{\!1}(u):=\{a\in[N]:|a-u|=1\}\).
Define the local ReLU penalty \(P_u(s):=\sum_{a\in\mathcal N_{\!1}(u)}
\relu(\langle p_a-p_u,s\rangle)\) for \(u\in[N]\) and \(s\in\R^2\).
This penalty has a separation property: if \(u\in\argmax_{a\in[N]}\langle p_a,s\rangle\), then \(P_u(s)=0\), while for \(r_c=(2c,-1)\), \(P_c(r_c)=0\) and \(P_u(r_c)\ge1\) whenever \(u\ne c\); see Appendix~\ref{app:cnn-penalty-proof}.
Let \(\Forb:=\{(u,v)\in[N]^2:u=v\text{ or }\{u,v\}\notin E\}\), and write
\(z=(z_1,\ldots,z_k)\in(\R^2)^k\), where \(z_i\) is associated with the \(i\)-th clique position.

The Fr\'echet-side loss is
\begin{equation}\label{eq:cnn-frechet-loss}
    L_G^{\mathrm F}(z):=
    \max_{\substack{1\le i<j\le k\\(u,v)\in\Forb}}
    \bigl[-P_u(z_i)-P_v(z_j)\bigr].
\end{equation}
The Clarke-side loss uses one additional filter coordinate \(t\) and is
\begin{equation}\label{eq:cnn-clarke-loss}
    L_G^{\mathrm C}(z,t):=
    \frac{t}{2}+\max\left\{L_G^{\mathrm F}(z),
    -\frac12\relu(t)-\frac12\relu(-t)\right\}
    =
    \frac{t}{2}+\max\left\{L_G^{\mathrm F}(z),-\frac{|t|}{2}\right\}.
\end{equation}

\begin{theorem}[Parameterized hardness for shallow CNN losses]\label{thm:cnn-hardness}
For the CNN family in~\eqref{eq:cnn-frechet-loss}--\eqref{eq:cnn-clarke-loss}, with stationarity tested at the origin of the trainable parameter space and parameterized by the trainable dimension \(p=\dim(\theta)\), the following hold.
\begin{enumerate}[label=\textnormal{(\roman*)},leftmargin=2.2em]
\item For \(L_G^{\mathrm F}\), for every fixed \(\eps\ge0\), Fr\'echet-\(\eps\)-NO is W[1]-hard with respect to \(p\).
\item For \(L_G^{\mathrm F}\), deciding whether the origin is not a local minimizer is W[1]-hard with respect to \(p\).
\item For \(L_G^{\mathrm C}\), for every fixed \(\eps\in[0,1/2)\), Clarke-\(\eps\)-YES is W[1]-hard with respect to \(p\).
\end{enumerate}
Moreover, under Hypothesis~\ref{hyp:clique}, none of these hard problems admits a deterministic algorithm running in time
\(\rho(p)M^{o(p)}\) for any computable function \(\rho\), where \(M\) denotes the CNN encoding size.
\end{theorem}

\begin{proof}[Proof sketch]
We reduce from \(k\)-Clique.
For the Fr\'echet side, use \(L_G^{\mathrm F}\) with trainable parameter
\(z\in(\R^2)^k\). If \(G\) has no \(k\)-clique, then for arbitrary \(z\)
choose \(v_i\in\argmax_{a\in[N]}\langle p_a,z_i\rangle\) for each clique
position \(i\). Since \((v_1,\ldots,v_k)\) is not a clique, some
\((v_i,v_j)\in\Forb\). For the term indexed by \((i,j,v_i,v_j)\), the
separation property (see Lemma~\ref{lem:cnn-penalty}) gives \(P_{v_i}(z_i)=P_{v_j}(z_j)=0\), while every term
in~\eqref{eq:cnn-frechet-loss} is nonpositive. Hence
\(L_G^{\mathrm F}\equiv0\) and \(\Fsub L_G^{\mathrm F}(0)=\{0\}\).

If \(G\) has a \(k\)-clique \((c_1,\ldots,c_k)\), take
\(z_i=T(2c_i,-1)\). For any term indexed by \((i,j,u,v)\), the pair
\((u,v)\) is forbidden whereas \((c_i,c_j)\) is allowed, so
\(u\ne c_i\) or \(v\ne c_j\). The separation property gives
\(P_u(z_i)\ge T\) in the first case or \(P_v(z_j)\ge T\) in the second,
and therefore \(L_G^{\mathrm F}(z)\le -T<0\). Since
\(L_G^{\mathrm F}\le0\) everywhere and is positively homogeneous, any
\(g\in\Fsub L_G^{\mathrm F}(0)\) would satisfy
\(\langle g,d\rangle\le L_G^{\mathrm F}(d)\) for all \(d\). Taking
\(d=g\) gives \(g=0\), but \(0\in\Fsub L_G^{\mathrm F}(0)\) would imply
\(L_G^{\mathrm F}\ge0\), contradicting the negative direction above.
Thus \(\Fsub L_G^{\mathrm F}(0)=\emptyset\), proving the Fr\'echet hard side.
Since \(L_G^{\mathrm F}\) is continuous PA, Lemma~\ref{lem:pa-local-model} also says that the no-clique and clique cases are exactly the local-minimum and non-local-minimum cases.

For the Clarke side, use \(L_G^{\mathrm C}\) with trainable parameter
\((z,t)\). If \(G\) has no \(k\)-clique, then \(L_G^{\mathrm F}\equiv0\),
so \(L_G^{\mathrm C}(z,t)=t/2\) and
\(\dist(0,\Csub L_G^{\mathrm C}(0,0))=1/2\). If \(G\) has a \(k\)-clique,
the Fr\'echet construction gives \(z^\star\) with
\(L_G^{\mathrm F}(z^\star)<-1\). Near \((z^\star,1)\), the
\(-|t|/2\) alternative is strictly selected in~\eqref{eq:cnn-clarke-loss},
so \(L_G^{\mathrm C}(z,t)=t/2-t/2=0\). By positive homogeneity, this yields a full-dimensional cone whose closure contains the origin on which
\(L_G^{\mathrm C}\) is affine with gradient zero; hence the zero gradient is essentially active and
\(0\in\Csub L_G^{\mathrm C}(0,0)\).
This proves Clarke-\(\eps\)-YES hardness for every fixed
\(\eps\in[0,1/2)\).

The CNN instances can be assembled in time \(f(k)|(G,k)|^{O(1)}\), with \(M=\poly(N,k)\) and \(p=\Theta(k)\), so these are valid parameterized reductions.  The running-time lower bounds follow from Hypothesis~\ref{hyp:clique}. Full details are in
Appendix~\ref{app:cnn-hardness-proof}.
\end{proof}

\begin{corollary}[Complete parameterized picture]\label{cor:cnn-complete}
For the CNN losses in~\eqref{eq:cnn-frechet-loss}--\eqref{eq:cnn-clarke-loss}, parameterized by the trainable dimension \(p\), all four stationarity polarities are decidable in time \(M^{O(p)}\), hence in XP.
\end{corollary}

The proof is deferred to Appendix~\ref{app:cnn-maxmin-xp-proof}.
Together with Theorem~\ref{thm:cnn-hardness}, this makes the \(M^{O(p)}\) upper bound tight under Hypothesis~\ref{hyp:clique} for Fr\'echet-\(\eps\)-NO and Clarke-\(\eps\)-YES in the stated ranges of \(\eps\).

Hence, the parameterized-complexity picture developed in
Sections~\ref{sec:fixed-d}--\ref{sec:hardness} carries over to stationarity testing for a simple family of CNN training losses. This reinforces the intractability of neural-network training from a local-verification perspective: even for shallow ReLU CNN losses, testing stationarity and local minimality can be W[1]-hard with respect to the number of trainable parameters.

\section{Conclusion}
This paper addresses a basic and foundational problem in nonsmooth optimization: deciding approximate first-order stationarity at a prescribed point for continuous PA functions.
For the input models considered here, we resolve the fixed-dimensional tractability picture from a parameterized-complexity viewpoint.
Fr\'echet-\(\eps\)-YES and Clarke-\(\eps\)-NO admit deterministic XP algorithms in the ambient dimension \(d\), and all four tests are in XP by complementing the output.
Conversely, Fr\'echet-\(\eps\)-NO for \(\eps\ge0\) and Clarke-\(\eps\)-YES for \(\eps\in[0,1/2)\) are W[1]-hard and admit no \(\rho(d)\size^{o(d)}\)-time deterministic algorithms under Hypothesis~\ref{hyp:clique}; for explicit max--min input, this matches the \(\size^{O(d)}\) enumeration dependence.

\textbf{Limitations.} Our results are worst-case complexity results for explicitly represented PA functions and a shallow CNN family.
The lower bounds do not rule out faster algorithms under additional geometric regularity, average-case assumptions, or more specialized neural architectures.


\subsection*{AI Disclosure}\label{par:ai-assistance}
The author conceived the overall research approach.
An AI assistant suggested the hard-instance construction for the $k$-Clique reduction in Section~\ref{sec:hardness} and helped develop its proof.
The assistant also helped formalize the author's proof ideas for the fixed-dimension algorithms in Section~\ref{sec:fixed-d} and the shallow CNN extension in Section~\ref{sec:cnn}.

\begin{ack}
Yuhan Ye is supported by the MIT Schwarzman College of Computing PhD Fellowship.
The author declares no competing financial interests.
\end{ack}

\bibliographystyle{plainnat}
\bibliography{references}

\clearpage
\appendix

\section{Supplementary material for Section~\ref{sec:prelim}}\label{app:section2-background}

\subsection{A geometric guide to Clarke and Fr\'echet stationarity}\label{app:stationarity-guide}
This subsection gives a geometric reading of the two stationarity notions used in the paper. The general definitions follow the standard nonsmooth-analysis viewpoint; see, for example, \citet{Clarke1990,RockafellarWets1998} and the monograph of \citet{CuiPang2021}. The key distinction is that Fr\'echet stationarity is governed by first-order local lower supports centered at the base point, whereas Clarke stationarity, for locally Lipschitz functions, is governed by the convex hull of limiting gradients from nearby differentiability points.

\paragraph{General definitions.}
For an extended-real function that is finite at $x$, the Fr\'echet, or regular, subdifferential can be written as
\[
 \Fsub f(x)=\left\{v\in\R^d:
 \liminf_{y\to x,\ y\ne x}
 \frac{f(y)-f(x)-\langle v,y-x\rangle}{\|y-x\|}\ge0\right\}.
\]
Thus $v\in\Fsub f(x)$ means that the affine function $y\mapsto f(x)+\langle v,y-x\rangle$ is a local lower support up to an $o(\|y-x\|)$ error.  In the convex case this reduces to the usual global supporting-hyperplane inequality.  For the continuous PA functions in this paper, the exact local fan model turns this into the directional inequality $\langle v,u\rangle\le f'(x;u)$ for every direction $u$.

For a locally Lipschitz function, Rademacher's theorem gives differentiability almost everywhere.  If $D_f$ denotes the set of differentiability points of $f$, the Clarke subdifferential has the limiting-gradient representation
\[
 \Csub f(x)=\cl\conv\left\{\lim_{r\to\infty}\nabla f(x_r):
     x_r\in D_f,\ x_r\to x,\ \nabla f(x_r)\text{ converges}\right\}.
\]
Equivalently, it is the compact convex set whose support function is the Clarke directional derivative $f^\circ(x;u)$.  For PA functions this formula becomes especially concrete (see Lemma \ref{lem:dc-clarke-slopes}): if $S_e(x)$ is the set of gradients of affine pieces attained on full-dimensional cells whose closures contain $x$, then
\[
    \Csub f(x)=\conv S_e(x).
\]
This is the precise sense in which Clarke stationarity checks whether the gradients of neighboring full-dimensional regions can be convexly balanced to zero.

If $f$ is differentiable at $x$, then both notions reduce to the usual gradient: $\Fsub f(x)=\Csub f(x)=\{\nabla f(x)\}$.  The distinction only appears at nonsmooth points, where the exact one-point directional model and the convexified nearby-gradient model can behave differently.

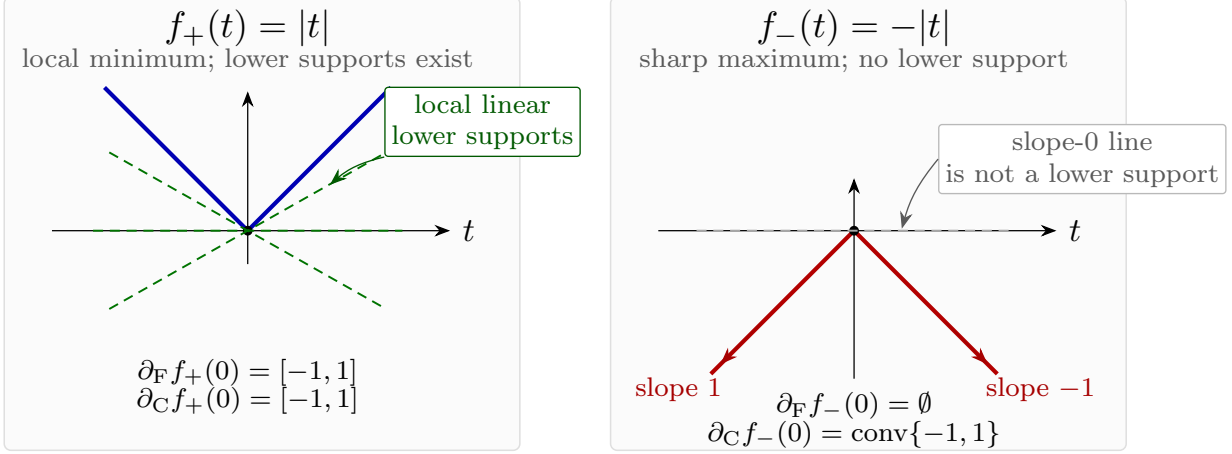
\begin{figure}[H]
\centering
\resizebox{0.98\linewidth}{!}{%
\begin{tikzpicture}[
    x=1cm,y=1cm,>=Stealth,
    panel/.style={rounded corners=3pt,fill=gray!2,draw=gray!22,line width=0.45pt},
    divider/.style={draw=gray!18,line width=0.4pt},
    axis/.style={->,draw=gray!65!black,line width=0.42pt},
    func/.style={draw=blue!60!black,line width=1.45pt,line cap=round,line join=round},
    support/.style={densely dashed,line width=0.68pt,line cap=round},
    legend/.style={font=\scriptsize,anchor=west},
    result/.style={font=\scriptsize,align=center},
    every node/.style={font=\small}
]
\begin{scope}[shift={(0,0)}]
  \draw[panel] (-2.55,-2.30) rectangle (2.85,2.45);
  \node[font=\scriptsize\bfseries,text=gray!60!black] at (-2.22,2.12) {(a)};
  \node[font=\bfseries] at (0.15,2.12) {$f_+(t)=|t|$};
  \draw[support,green!45!black] (-1.86,1.70) -- (-1.32,1.70);
  \node[legend,text=green!38!black] at (-1.18,1.70)
    {examples of lower supports at $0$};

  \draw[axis] (-2.05,0) -- (2.12,0) node[right] {$t$};
  \draw[axis] (0,-1.34) -- (0,1.34);
  \draw[func] (-1.30,1.30) -- (0,0) -- (1.30,1.30);

  \draw[support,green!45!black] (-1.45,-0.82) -- (1.45,0.82);
  \draw[support,green!45!black] (-1.62,0) -- (1.62,0);
  \draw[support,green!45!black] (-1.45,0.82) -- (1.45,-0.82);
  \fill[blue!60!black] (0,0) circle (1.5pt);

  \draw[divider] (-2.32,-1.47) -- (2.62,-1.47);
  \node[result] at (0.15,-1.86)
    {$\Fsub f_+(0)=[-1,1]$\\[-1pt]$\Csub f_+(0)=[-1,1]$};
\end{scope}

\begin{scope}[shift={(6.35,0)}]
  \draw[panel] (-2.55,-2.30) rectangle (2.85,2.45);
  \node[font=\scriptsize\bfseries,text=gray!60!black] at (-2.22,2.12) {(b)};
  \node[font=\bfseries] at (0.15,2.12) {$f_-(t)=-|t|$};
  \draw[support,orange!78!black] (-1.72,1.70) -- (-1.18,1.70);
  \node[legend,text=orange!65!black] at (-1.04,1.70)
    {not a lower support};

  \draw[axis] (-2.05,0) -- (2.12,0) node[right] {$t$};
  \draw[axis] (0,-1.34) -- (0,1.34);
  \draw[func] (-1.30,-1.30) -- (0,0) -- (1.30,-1.30);

  \draw[support,orange!78!black] (-1.62,0) -- (1.62,0);
  \fill[blue!60!black] (0,0) circle (1.5pt);

  \draw[divider] (-2.32,-1.47) -- (2.62,-1.47);
  \node[result] at (0.15,-1.86)
    {$\Fsub f_-(0)=\emptyset$\\[-1pt]$\Csub f_-(0)=[-1,1]$};
\end{scope}
\end{tikzpicture}%
}

\caption{Fr\'echet stationarity is controlled by first-order local lower supports of the exact local model, whereas Clarke stationarity uses the convex hull of gradients from nearby full-dimensional differentiability regions.  The functions $|t|$ and $-|t|$ have the same set of neighboring slopes, $\{-1,1\}$, at the origin, but only $|t|$ is Fr\'echet stationary.}
\label{fig:frechet-clarke-guide}
\end{figure}

\paragraph{A one-dimensional example.}
At the origin, $f_+(t)=|t|$ has $f_+'(0;u)=|u|$, so
\[
    \Fsub f_+(0)=\{g\in\R: gu\le |u|\ \forall u\}= [-1,1],
    \qquad
    \Csub f_+(0)=[-1,1].
\]
Both stationarity notions hold, and the origin is a local minimizer.  By contrast, $f_-(t)=-|t|$ has $f_-'(0;u)=-|u|$.  The inequalities $gu\le-|u|$ would require $g\le -1$ from $u>0$ and $g\ge1$ from $u<0$, so $\Fsub f_-(0)=\emptyset$.  Nevertheless the two neighboring full-dimensional gradients are still $-1$ and $1$, hence $\Csub f_-(0)=\conv\{-1,1\}=[-1,1]$ and $0\in\Csub f_-(0)$.  This example is the simplest reason why Clarke stationarity may hold at a sharp maximum, while exact Fr\'echet stationarity for continuous PA functions is equivalent to local minimality.

\paragraph{A nonempty strict-inclusion example.}
The inclusion $\Fsub f(x)\subseteq\Csub f(x)$ can be strict even when the Fr\'echet subdifferential is nonempty.  Consider the positively homogeneous PA function
\[
    f(x_1,x_2):=\max\{0,\min\{x_1,x_2\}\}.
\]
At the origin, the full-dimensional affine pieces have gradients $0$, $e_1$, and $e_2$, where $e_1,e_2$ are the standard basis vectors.  Hence
\[
    \Csub f(0)=\conv\{0,e_1,e_2\}.
\]
For the Fr\'echet subdifferential, the directional inequality $\langle v,u\rangle\le f(u)$ for all $u$ applies because $f$ is positively homogeneous.  Since $f(1,0)=f(-1,0)=f(0,1)=f(0,-1)=0$, any $v=(v_1,v_2)\in\Fsub f(0)$ must satisfy both $v_i\le0$ and $v_i\ge0$ for $i=1,2$, so $v=0$.  Conversely $0\in\Fsub f(0)$ because $f\ge0$.  Therefore
\[
    \Fsub f(0)=\{0\}\subsetneq \conv\{0,e_1,e_2\}=\Csub f(0).
\]
This example separates the two notions without relying on an empty Fr\'echet subdifferential.

\subsection{Deferred proofs for local PA calculus}\label{app:prelim-proofs}

\begin{proof}[Proof of Lemma~\ref{lem:pa-local-model}]
Take a finite polyhedral partition on which \(f\) is affine.  After translating the
base point to the origin, each cell whose closure contains \(w\) is locally described by
the inequalities of that cell that are tight at \(w\); inequalities not tight at \(w\)
remain inactive in a sufficiently small ball.  Hence, for some \(r>0\), the cells incident
to \(w\) induce a finite polyhedral fan in the displacement variable \(u\) on
\(\{u:\|u\|\le r\}\).

On each cone of this fan, \(f(w+u)-f(w)\) is a linear function of \(u\).  Because \(f\)
is continuous, these linear pieces agree on common faces, and therefore define a
continuous positively homogeneous PA function \(\phi_w\) satisfying
\[
    f(w+u)=f(w)+\phi_w(u)\qquad(\|u\|\le r).
\]
Consequently,
\[
    f'(w;u)=\lim_{t\downarrow0}\frac{f(w+tu)-f(w)}{t}
           =\lim_{t\downarrow0}\frac{\phi_w(tu)}{t}
           =\phi_w(u).
\]
The exact local identity also identifies the Fr\'echet and Clarke subdifferentials of
\(f\) at \(w\) with those of \(\phi_w\) at the origin, since both definitions depend only
on the function in an arbitrarily small neighborhood of the base point.

Finally, \(0\in\Fsub f(w)\) is equivalent to \(f'(w;u)=\phi_w(u)\ge0\) for every
direction \(u\).  By the exact local identity, this is equivalent to
\(f(w+u)\ge f(w)\) for all sufficiently small \(u\), namely to \(w\) being a local
minimizer of \(f\).
\end{proof}

\begin{proof}[Proof of Lemma~\ref{lem:frechet-translation}]
For convex polyhedral functions, directional derivatives are support functions of their subdifferentials: $h'(w;u)=\sigma_X(u):=\max_{x\in X}\langle x,u\rangle$ and $g'(w;u)=\sigma_Y(u):=\max_{y\in Y}\langle y,u\rangle$.
Thus $(h-g)'(w;u)=\sigma_X(u)-\sigma_Y(u)$, and
\[
 s\in\Fsub(h-g)(w)
 \iff
 \langle s,u\rangle+\sigma_Y(u)\le\sigma_X(u)\quad\forall u\in\R^d.
\]
Because $\sigma_{s+Y}(u)=\langle s,u\rangle+\sigma_Y(u)$, this is support-function domination $\sigma_{s+Y}\le\sigma_X$.
For compact convex sets, support-function domination is equivalent to containment, so $s+Y\subseteq X$.
The approximate statement is obtained by requiring some $s\in\Fsub(h-g)(w)$ with $\|s\|\le\eps$, equivalently $s\in\eps\Ball$.
Setting $\eps=0$ gives the exact stationarity condition $Y\subseteq X$.
\end{proof}

\begin{proof}[Proof of Lemma~\ref{lem:dc-clarke-slopes}]
Let \(S_e(X,Y)\) be the set in the statement.  For a vertex \(v\) of a compact polytope
\(P\), the relative interior of its normal cone is
\[
 \ri N_P(v)
 =
 \{u:\langle v,u\rangle>\langle v',u\rangle\
     \text{for every }v'\in\vertices(P)\setminus\{v\}\}.
\]
Thus \(x\in\vertices(X)\) and \(y\in\vertices(Y)\) are simultaneously exposed exactly
when there is a direction \(u\) satisfying
\[
 \langle x-x',u\rangle>0\quad(x'\in\vertices(X)\setminus\{x\}),
 \qquad
 \langle y-y',u\rangle>0\quad(y'\in\vertices(Y)\setminus\{y\}).
\]
On such a cone, both support functions are differentiable and
\[
    \phi(u)=\sigma_X(u)-\sigma_Y(u)=\langle x-y,u\rangle,
\]
so \(x-y\) is a slope attained on a neighboring full-dimensional differentiability
region of \(\phi\).

Conversely, consider any full-dimensional cone on which \(\phi\) is affine.  Refining
this cone by the common refinement of the normal fans of \(X\) and \(Y\), we obtain a
full-dimensional subcone on which both \(\sigma_X\) and \(\sigma_Y\) are affine.  An
interior direction of this subcone uniquely exposes some vertices \(x\in\vertices(X)\)
and \(y\in\vertices(Y)\), and the slope of \(\phi\) on the subcone is \(x-y\).  Hence
\(S_e(X,Y)\) is exactly the set of slopes attained on neighboring full-dimensional
differentiability regions of the PA function \(\phi\).  The Clarke subdifferential of
a PA function is the convex hull of these essentially active slopes, so
\[
    \Csub\phi(0)=\conv(S_e(X,Y)).
\]
\end{proof}

\subsection{A polynomial baseline for single-max DC input}\label{app:single-max-dc}

\begin{lemma}[Polynomial tests for single-max DC input]\label{lem:single-max-dc-p}
Let
\[
 h(w)=\max_{i\in I}\{\langle a_i,w\rangle+\alpha_i\},
 \qquad
 g(w)=\max_{j\in J}\{\langle b_j,w\rangle+\beta_j\},
\]
with rational data, and let $w\in\Q^d$ and $\eps\ge0$ be given.
Then both tests $0\in \Fsub(h-g)(w)+\eps\Ball$ and $0\in \Csub(h-g)(w)+\eps\Ball$ can be decided in polynomial time.
\end{lemma}

\begin{proof}
Let
\[
 A:=\{a_i:\ i\in I,\ \langle a_i,w\rangle+\alpha_i=h(w)\},
 \qquad
 B:=\{b_j:\ j\in J,\ \langle b_j,w\rangle+\beta_j=g(w)\}
\]
be the active slope sets at $w$, after removing duplicates.
Then $\Csub h(w)=\conv(A)$ and $\Csub g(w)=\conv(B)$.
For the Fr\'echet test, Lemma~\ref{lem:frechet-translation} gives $0\in\Fsub(h-g)(w)+\eps\Ball$ if and only if there exists $s$ with $\|s\|\le\eps$ and $s+\conv(B)\subseteq\conv(A)$.
Since $\conv(B)$ is the convex hull of the explicitly listed points in $B$, the containment condition is equivalent to $s+b\in\conv(A)$ for every $b\in B$.
Thus the test is the feasibility of the polynomial-size system
\[
 s+b=\sum_{a\in A}\lambda_{b,a}a,
 \qquad
 \sum_{a\in A}\lambda_{b,a}=1,
 \qquad
 \lambda_{b,a}\ge0\quad (b\in B),
 \qquad
 \|s\|\le\eps.
\]
which is a rational convex QP; for $\eps=0$ it reduces to a linear feasibility problem.

For the Clarke test, the local directional model is $\phi(u)=\max_{a\in A}\langle a,u\rangle-\max_{b\in B}\langle b,u\rangle$.
A pair $(a,b)\in A\times B$ contributes the slope $a-b$ on a full-dimensional linearity cone precisely when the strict system
\[
 \langle a-a',u\rangle>0\quad(a'\in A\setminus\{a\}),
 \qquad
 \langle b-b',u\rangle>0\quad(b'\in B\setminus\{b\})
\]
is feasible.
By homogeneity, this can be checked by the linear feasibility system with all right-hand sides replaced by $1$.
Solving this LP for every pair $(a,b)$ gives the essentially active slope set $S_e$ in polynomial time, and $\Csub(h-g)(w)=\conv(S_e)$.
Finally, $0\in \Csub(h-g)(w)+\eps\Ball$ if and only if the minimum of $\|\sum_{s\in S_e}\lambda_s s\|$ over $\lambda\ge0$ and $\sum_{s\in S_e}\lambda_s=1$ is at most $\eps$, which is a convex quadratic program of polynomial size.
\end{proof}

\paragraph{Local geometry and tractable structure.}
Explicit local geometry also yields tractable cases: if the full-dimensional cones and gradients of the local PA model at the query point are explicitly listed, the two tests reduce to polynomial-size linear or convex-quadratic programs. For DC input $f=h-g$, set $X=\Csub h(w)$ and $Y=\Csub g(w)$. Compatibility of $X,Y$ characterizes when $\Csub(h-g)(w)=X-Y$; the transversality condition $\operatorname{par}(X)\cap\operatorname{par}(Y)=\{0\}$ is a polynomial-time verifiable sufficient condition in the explicit $n$-$\MC$ model, where $\operatorname{par}(P)=\operatorname{span}(P-P)$~\citep[Theorem~4.11 and Propositions~4.16, 4.18]{TianSo2025}. When this equality holds and $X,Y$ are explicitly available, Clarke testing reduces to convex quadratic programming over $X-Y$.

\section{Deferred proofs for Theorem~\ref{thm:fixed-d-xp}}\label{app:fixed-d-proofs}

To prove Theorem~\ref{thm:fixed-d-xp}, we first establish a vertex-enumeration bound for the convex subdifferential polytopes arising from fixed-depth $n$-$\MC$ DC input. At a max node, the subdifferential is the convex hull of the union of the subdifferentials of the children attaining the maximum at the query point; at a sum node, it is the Minkowski sum of all child subdifferentials. Our main task is therefore to enumerate the vertices produced by these two operations and bound both the output size and running time by a power of the input length whose exponent is computable.

\paragraph{Background.}
Candidate vertices of a convex hull can be pruned by linear programming. For Minkowski sums, vertex enumeration is available~\citep[Theorem~3.3]{Fukuda2004}; see also \citet{Weibel2012} for fixed-dimensional vertex bounds in terms of the total number of input vertices. The proof below gives the bound needed here by enumerating a hyperplane arrangement. An arrangement of $N$ hyperplanes in $\R^d$ has $N^{O(d)}$ cells, and its cell complex can be constructed in $N^{O(d)}$ time for fixed $d$~\citep{Zaslavsky1975,EdelsbrunnerORourkeSeidel1986}. For background on normal cones, exposed faces, and polyhedral computation, see \citet{Ziegler1995,KaibelPfetsch2003,Fukuda2020}.

The following lemma makes this precise.
\begin{lemma}[Vertex enumeration for $n$-$\MC$ subdifferentials]\label{lem:mc-v-growth}
Fix an ambient dimension $d$ and an integer depth bound $D$.
Let $F:\R^d\to\R$ be a convex PA expression, represented by a rational expression tree of depth
at most $D$, whose leaves are affine functions and whose internal nodes are finite sum nodes and
finite max nodes.  Let $\size$ denote the total binary encoding length of the expression and of
the query point $w$.
For every node $\nu$, let $F_\nu$ be the subexpression rooted at $\nu$ and put
\[
    P_\nu:=\Csub F_\nu(w).
\]
Then there is a computable function $\Gamma_D(d)$, depending only on $d$ and the fixed depth
bound $D$, such that for every node $\nu$ the vertex set $\vertices(P_\nu)$ has cardinality at most
$\size^{\Gamma_D(d)}$, all listed vertices have rational bit length at most
$\size^{\Gamma_D(d)}$, and the vertex sets $\vertices(P_\nu)$ for all nodes can be computed in
deterministic time $\size^{\Gamma_D(d)}$.
\end{lemma}

\begin{proof}
Let the expression be represented by a rooted tree.  For a node \(\nu\), write
\(\operatorname{child}(\nu)\) for the set of its children, and define the height of \(\nu\)
as the maximum length of a downward path from \(\nu\) to a leaf.  Leaves have height zero.
We prove the claim by induction on node height.  We slightly strengthen the induction statement:
for every node of height at most \(t\), we compute both the exact value \(F_\nu(w)\) and the
vertex list \(\vertices(P_\nu)\), where \(P_\nu=\Csub F_\nu(w)\), with cardinalities and bit lengths
bounded by \(\size^{\Gamma_t(d)}\).

We first record the fixed-dimensional vertex-enumeration fact used in the induction.  Let
\(P_1,\dots,P_R\subseteq\R^d\) be rational polytopes given by reduced vertex lists, and let \(N\)
be the total number of input vertices.  For each fixed \(d\), there is a constant \(c_d\) such that
the vertices of
\[
    \conv\Bigl(\bigcup_{r=1}^R P_r\Bigr)
    \qquad\text{and}\qquad
    \sum_{r=1}^R P_r
\]
can be computed in time \(N^{c_d}\) times a polynomial in the input bit length.

For the convex hull of a union, every extreme point of \(\conv(\cup_r P_r)\) is already a vertex of at least one input polytope.  Thus the union of the input vertex lists is a complete candidate list.  To reduce it, test each candidate \(v\) for membership in the convex hull of the other candidates; this is a rational linear feasibility problem, and exactly the candidates that fail the test are the vertices of the convex hull.

For the Minkowski sum \(P:=P_1+\cdots+P_R\), we use exposed faces.  For a polytope \(P\) and a direction \(u\), define
\[
    F_P(u):=\argmax_{x\in P}\langle x,u\rangle .
\]
This is the face on which the linear objective \(x\mapsto\langle x,u\rangle\) is maximized.  The directions exposing the same face form a normal cone, and these normal cones partition direction space.  We only need the following elementary consequence: for every direction \(u\),
\[
    F_{P_1+\cdots+P_R}(u)=F_{P_1}(u)+\cdots+F_{P_R}(u),
\]
because maximizing a linear function over a Minkowski sum separates over the summands.  Hence, if \(F_{P_r}(u)=\{v_r\}\) for every \(r\), then \(F_P(u)=\{v_1+\cdots+v_R\}\), so \(v_1+\cdots+v_R\) is a vertex of \(P\).

To find all such generic directions, form the hyperplanes
\[
    \langle v-v',u\rangle=0
\]
for all pairs of distinct vertices \(v,v'\) from the same summand \(P_r\).  These hyperplanes divide the direction space into open polyhedral cones.  On any one of these cones, the ordering of the values \(\langle v,u\rangle\) is fixed inside each summand, so each \(P_r\) has the same unique maximizing vertex throughout the cone.  The arrangement has \(O(N^2)\) hyperplanes and therefore \(N^{O(d)}\) full-dimensional cones in fixed dimension; by the fixed-dimensional arrangement construction, these cones can be enumerated in \(N^{O(d)}\) time~\citep{Zaslavsky1975,EdelsbrunnerORourkeSeidel1986}.

For each cone, choose a rational representative direction \(u\), compute the unique maximizers \(v_r(u)\in\vertices(P_r)\), and add \(v_1(u)+\cdots+v_R(u)\) to the candidate list.  Remove duplicate candidates.  Then test each remaining candidate \(q\) for membership in the convex hull of the other candidates; this is a rational linear feasibility problem, and the candidates that fail this test are exactly the vertices.  The candidate list is complete: if \(q\) is a vertex of \(P\), choose \(u\) in the interior of its normal cone, so \(F_P(u)=\{q\}\).  Since finitely many hyperplanes have empty interior, we may choose such a \(u\) outside all hyperplanes above.  Then each summand has a unique maximizer \(v_r(u)\), and
\[
    q=v_1(u)+\cdots+v_R(u),
\]
so \(q\) appears in the list.  Since the candidates are sums of input vertices and the filtering steps are rational LP feasibility tests, the running time is \(N^{O(d)}\) times a polynomial in the input bit length, and all intermediate bit lengths stay polynomially bounded.

For a leaf \(F_\nu(x)=\langle a,x\rangle+b\), we have \(P_\nu=\{a\}\), and both \(F_\nu(w)\) and
\(\vertices(P_\nu)\) are computed directly.  Thus the claim holds for height zero.

Now assume the claim holds for all nodes of height at most \(t-1\), and let \(\nu\) be a node of
height \(t\).  Since the expression is explicitly listed, every node has at most \(\size\) children,
and the total number of nodes is at most \(\size\).

If \(\nu\) is a sum node, say
\[
    F_\nu=\sum_{\mu\in\operatorname{child}(\nu)}F_\mu,
\]
then \(F_\nu(w)=\sum_{\mu\in\operatorname{child}(\nu)}F_\mu(w)\), and convex subdifferential
calculus gives
\[
    P_\nu=\sum_{\mu\in\operatorname{child}(\nu)}P_\mu .
\]
By the induction hypothesis, the total number of child vertices is at most
\(\size\cdot\size^{\Gamma_{t-1}(d)}\).  Applying the fixed-dimensional Minkowski-sum enumeration
above gives
\[
    |\vertices(P_\nu)|
    \le
    \bigl(\size^{\Gamma_{t-1}(d)+1}\bigr)^{c_d}
    \le
    \size^{\Gamma_t(d)}
\]
after choosing \(\Gamma_t(d)\) large enough.  The bit length of \(F_\nu(w)\) and of the listed
vertices satisfies the same type of bound, after increasing \(\Gamma_t(d)\) if necessary.

If \(\nu\) is a max node, say
\[
    F_\nu=\max_{\mu\in\operatorname{child}(\nu)}F_\mu,
\]
we first compute all child values \(F_\mu(w)\) exactly and determine the active set
\[
    \operatorname{Act}(\nu,w)
    :=
    \{\mu\in\operatorname{child}(\nu):F_\mu(w)=F_\nu(w)\}.
\]
The exact comparisons have bit length bounded by the induction hypothesis.  Convex
subdifferential calculus for a pointwise maximum gives
\[
    P_\nu
    =
    \conv\Bigl(\bigcup_{\mu\in\operatorname{Act}(\nu,w)}P_\mu\Bigr).
\]
The candidate vertices are contained in the union of the active child vertex lists, whose total
size is at most \(\size^{\Gamma_{t-1}(d)+1}\).  Removing non-extreme candidates by rational linear
feasibility tests gives \(\vertices(P_\nu)\) within the same bound, again after increasing
\(\Gamma_t(d)\) if necessary.

Thus, for every height \(t\), there is a computable exponent \(\Gamma_t(d)\) controlling the
vertex-list size, bit length, and computation time for all nodes of height at most \(t\).  Since the
depth bound \(D\) is fixed independently of the input, taking \(t=D\) gives a computable exponent
\(\Gamma_D(d)\).  The total work over all nodes is absorbed into \(\size^{\Gamma_D(d)}\).
\end{proof}

\begin{remark}
For explicit \(n\)-\(\MC\) DC input, the lifted LP descriptions of
\citet[Appendix~A]{TianSo2025} give polynomial-size
projected representations of the subdifferentials.  This is used to obtain polynomial-time subroutines
for subgradient membership, the sum-rule relaxation condition
\(0\in\partial h(w)-\partial g(w)+\eps B\), and transversality
\citep[Proposition~A.1(a)--(c)]{TianSo2025}.  These subroutines also support the (co)NP membership proofs for exact stationarity testing
\citep[Theorem~3.2]{TianSo2025}

Lemma~\ref{lem:mc-v-growth}, by contrast, uses a different tool tailored to our
fixed-dimensional parameterized setting.  Under a fixed depth bound \(D\), the
same subdifferential polytopes are converted into explicit
reduced vertex lists by an efficient vertex-enumeration algorithm
routine~\citep{EdelsbrunnerORourkeSeidel1986,Ziegler1995}, within
\(\size^{\Gamma_D(d)}\) time.  This use of vertex enumeration is the new
algorithmic ingredient here.
\end{remark}

We also use the classical polynomial-time solvability of rational convex quadratic
programming in the Turing model~\citep{KozlovTarasovKhachiyan1980}; see also
\citet{GrotchelLovaszSchrijver1988} for the rational-data ellipsoid framework.
We record only the projection forms needed below.

\begin{lemma}[Exact polynomial-time comparison for rational quadratic subproblems]\label{lem:exact-rational-qp}
Let \(Q\succeq0\), \(q\), \(r\), \(A\), \(b\), and \(\tau\) be rational input data.
The threshold decision problem
\[
    \inf\{z^\top Qz+2q^\top z+r:\ Az\le b\}\le \tau
\]
can be decided in polynomial time in the binary encoding length of the displayed rational
data.  The same algorithmic framework detects infeasibility and unboundedness below.

In particular, the following two projection comparisons are exact polynomial-time decisions
in their encoded size:
\[
    \min_{z\in P}\|Lz+c\|_2^2\le \tau
    \quad\text{for a rational polyhedron }P,
\]
and
\[
    \min\left\{
        \left\|\sum_{a\in S}\lambda_a a\right\|_2^2:
        \lambda_a\ge0,\ 
        \sum_{a\in S}\lambda_a=1
    \right\}\le \tau
    \quad\text{for a finite nonempty rational set }S\subseteq\R^d .
\]
Here all matrices, vectors, points, and the threshold \(\tau\) are rationally encoded.
\end{lemma}

\begin{proof}
The first displayed problem is a convex quadratic program with rational data.  By the
polynomial-time solvability theorem for rational convex quadratic programming
of \citet{KozlovTarasovKhachiyan1980}, one can decide the rational threshold comparison
in the Turing model in time polynomial in the binary input length; infeasibility and
unboundedness below are detected within the same bound.

The two projection problems are special cases.  In the first one,
\[
    \|Lz+c\|_2^2
    =
    z^\top L^\top Lz+2(L^\top c)^\top z+\|c\|_2^2,
\]
so the objective has a rational positive semidefinite quadratic matrix \(L^\top L\), and
the constraints defining \(P\) are rational linear inequalities.  In the second one, the
variables are the simplex coefficients \(\lambda\), the feasible region is the rational
simplex, and the objective matrix is the rational Gram matrix
\[
    \bigl(\langle a,a'\rangle\bigr)_{a,a'\in S}.
\]
Thus both are rational convex quadratic programs, and their comparisons with the rational
threshold \(\tau\) are exact polynomial-time decisions in the encoded input size.
\end{proof}

\begin{proof}[Proof of Theorem~\ref{thm:fixed-d-xp}: fixed-depth explicit $n$-$\MC$ DC input, Fr\'echet-$\eps$-YES]
Let the fixed-depth explicit $n$-$\MC$ formulas for $h$ and $g$ be viewed as their canonical
alternating max/sum expression trees.  Apply Lemma~\ref{lem:mc-v-growth} to both trees.
This gives reduced vertex lists for
\[
    X:=\Csub h(w),
    \qquad
    Y:=\Csub g(w)
\]
in time $\size^{\Gamma(d)}$, where $\Gamma$ depends only on the fixed depth and on $d$.

By Lemma~\ref{lem:frechet-translation},
\[
    0\in\Fsub(h-g)(w)+\eps\Ball
    \iff
    \exists \xi\in\R^d \text{ such that } \|\xi\|\le\eps
    \text{ and } \xi+Y\subseteq X.
\]
Write
\[
    \vertices(X)=\{x_1,\dots,x_N\},
    \qquad
    \vertices(Y)=\{y_1,\dots,y_m\}.
\]
Since $Y=\conv\{y_1,\dots,y_m\}$ and $X=\conv\{x_1,\dots,x_N\}$, the containment
$\xi+Y\subseteq X$ is equivalent to requiring $\xi+y_i\in X$ for every $i\in[m]$.
Equivalently, for every $i\in[m]$, there exist coefficients $\lambda_{i,1},\dots,\lambda_{i,N}$
such that
\[
    \xi+y_i=\sum_{r=1}^N \lambda_{i,r}x_r,
    \qquad
    \sum_{r=1}^N\lambda_{i,r}=1,
    \qquad
    \lambda_{i,r}\ge0\quad(r=1,\dots,N).
\]
Thus the Fr\'echet-$\eps$-YES test is the exact comparison
\[
    \min_{\xi,\lambda}
    \left\{
        \|\xi\|_2^2:
        \xi+y_i=\sum_{r=1}^N \lambda_{i,r}x_r,
        \sum_{r=1}^N\lambda_{i,r}=1,
        \lambda_{i,r}\ge0
        \quad(i=1,\dots,m)
    \right\}
    \le \eps^2 .
\]
The number of variables, constraints, and coefficient bit lengths in the displayed rational convex
QP are bounded by $\size^{\Gamma(d)}$ after increasing $\Gamma$ if necessary.  By
Lemma~\ref{lem:exact-rational-qp}, this comparison is decided exactly in time polynomial in the
displayed QP size, hence in time $\size^{\Gamma(d)}$.
For local minimality, set \(\eps=0\) in the same Fr\'echet computation and use Lemma~\ref{lem:pa-local-model}, which identifies \(0\in\Fsub f(w)\) with \(w\) being a local minimizer for continuous PA functions.
\end{proof}

\begin{proof}[Proof of Theorem~\ref{thm:fixed-d-xp}: fixed-depth explicit $n$-$\MC$ DC input, Clarke-$\eps$-NO]
First compute reduced vertex lists for
\[
    X:=\Csub h(w),
    \qquad
    Y:=\Csub g(w)
\]
using Lemma~\ref{lem:mc-v-growth}.  For convex PA functions, the local directional model of
$f=h-g$ at $w$ is
\[
    \phi(u):=\sigma_X(u)-\sigma_Y(u),
\]
and Lemma~\ref{lem:pa-local-model} gives $\Csub f(w)=\Csub\phi(0)$.

Let $V_X:=\vertices(X)$ and $V_Y:=\vertices(Y)$.
For $x\in V_X$ and $y\in V_Y$, the pair $(x,y)$ is simultaneously exposed precisely when the
strict system
\[
    \langle x-x',u\rangle>0\qquad(x'\in V_X\setminus\{x\}),
\]
\[
    \langle y-y',u\rangle>0\qquad(y'\in V_Y\setminus\{y\})
\]
is feasible.  If the system has at least one inequality, homogeneity makes this equivalent to
feasibility of the rational linear system obtained by replacing every strict right-hand side by $1$.
If there are no inequalities, the pair is feasible by convention.  This formulation also covers
lower-dimensional polytopes and singleton polytopes: the strict vertex inequalities characterize
the relative interiors of the corresponding normal cones, and singleton normal cones contribute
no strict inequalities.

Define
\[
    S_e:=
    \{x-y:\ x\in V_X,\ y\in V_Y,
    \ri N_X(x)\cap\ri N_Y(y)\ne\emptyset\}.
\]
The number of vertex pairs and the size of every linear feasibility system are bounded by
$\size^{\Gamma(d)}$.  By Lemma~\ref{lem:dc-clarke-slopes},
\[
    \Csub f(w)=\Csub\phi(0)=\conv(S_e).
\]
Therefore Clarke-$\eps$-NO is decided by the exact comparison
\[
    \min\left\{
        \left\|\sum_{a\in S_e}\lambda_a a\right\|_2^2:
        \lambda_a\ge0,\ \sum_{a\in S_e}\lambda_a=1
    \right\}
    >\eps^2 .
\]
By Lemma~\ref{lem:exact-rational-qp}, this rational convex-QP comparison is exact and polynomial
in the displayed QP size.  The overall running time is $\size^{\Gamma(d)}$.
\end{proof}

\begin{proof}[Proof of Theorem~\ref{thm:fixed-d-xp}: explicit max--min input, Fr\'echet-$\eps$-YES]
Let
\[
    \phi(u)=\max_{i\in I(w)}\min_{j\in J_i(w)}\langle x_j,u\rangle
\]
be the exact local model from~\eqref{eq:local-model}.  By Lemma~\ref{lem:pa-local-model},
$\Fsub f(w)=\Fsub\phi(0)$.
Let $A$ be the set of distinct slopes appearing in this local model; clearly $|A|\le \size$.

Build the central hyperplane arrangement
\[
    \mathcal H
    :=
    \{\langle a-a',u\rangle=0:\ a,a'\in A,\ a\ne a'\}.
\]
If $\mathcal H$ is empty, it has the single full-dimensional cone $\R^d$.
Otherwise, it has $O(|A|^2)$ hyperplanes and hence $\size^{O(d)}$ full-dimensional open cones in
fixed dimension. With $H:=|\mathcal H|$, their number is at most $\sum_{j=0}^d\binom{H}{j}\le\size^{O(d)}$. These cones can be enumerated in time $\size^{O(d)}$~\citep{EdelsbrunnerORourkeSeidel1986}.

On each full-dimensional open cone $C$, the ordering of all quantities
$\langle a,u\rangle$, $a\in A$, is fixed and strict.  Consequently, every inner minimum and the
outer maximum select fixed slopes.  Thus there is a slope $a_C\in A$ such that
\[
    \phi(u)=\langle a_C,u\rangle
    \qquad(u\in C).
\]
Let
\[
    K_C:=\overline C.
\]
From the arrangement description, orient the inequalities so that
\[
    K_C=\{u:B_Cu\ge0\},
\]
where $B_C$ has $\size^{O(1)}$ rows.

By the directional-derivative definition of the Fr\'echet subdifferential and the positive homogeneity of \(\phi\), for a fixed \(\xi\) we have
\[
    \xi\in\Fsub\phi(0)
    \quad\Longleftrightarrow\quad
    \langle \xi,u\rangle\le\phi(u)\quad\text{for every }u\in\R^d .
\]
The closures \(K_C\) of the full-dimensional arrangement cones cover \(\R^d\), so this condition is equivalent to imposing the same inequality on every \(K_C\).  Moreover, \(\phi\) equals \(\langle a_C,u\rangle\) on the open cone \(C\), and by continuity the same equality extends to its closure \(K_C\).  Hence
\[
    \Fsub\phi(0)
    =
    \bigcap_C
    \{\xi:\langle \xi,u\rangle\le \phi(u)\ \forall u\in K_C\}.
\]
On $K_C$, this condition is
\[
    \langle \xi-a_C,u\rangle\le0\qquad\forall u\in K_C,
\]
or equivalently
\[
    \xi-a_C\in K_C^\circ .
\]
Because $K_C=\{u:B_Cu\ge0\}$, its polar cone is
\[
    K_C^\circ=\{-B_C^\top\lambda_C:\lambda_C\ge0\}.
\]
Therefore $\xi\in\Fsub\phi(0)$ if and only if, for every full-dimensional arrangement cone $C$,
there exists $\lambda_C\ge0$ such that
\[
    \xi=a_C-B_C^\top\lambda_C.
\]
Thus $0\in\Fsub f(w)+\eps\Ball$ is equivalent to feasibility of
\[
    \exists \xi,\{\lambda_C\}_C:\qquad
    \xi=a_C-B_C^\top\lambda_C\quad\forall C,
    \qquad
    \lambda_C\ge0\quad\forall C,
    \qquad
    \|\xi\|_2\le\eps .
\]
Equivalently, one minimizes $\|\xi\|_2^2$ over the displayed rational polyhedron in the variables
$\xi$ and $\{\lambda_C\}_C$ and compares the optimum with $\eps^2$.
The number of variables, constraints, and the bit length of all rational coefficients are bounded by
$\size^{O(d)}$.  Lemma~\ref{lem:exact-rational-qp} gives an exact comparison in time polynomial in
this displayed size, and hence in time $\size^{O(d)}$.
For local minimality, set \(\eps=0\) and apply Lemma~\ref{lem:pa-local-model}; the resulting exact Fr\'echet test decides whether \(w\) is a local minimizer.
\end{proof}

\begin{proof}[Proof of Theorem~\ref{thm:fixed-d-xp}: explicit max--min input, Clarke-$\eps$-NO]
Compute the local model $\phi$ as in the preceding proof.  By Lemma~\ref{lem:pa-local-model},
\[
    \Csub f(w)=\Csub\phi(0).
\]
Let $A$ be the set of distinct slopes appearing in $\phi$, and construct the central arrangement
\[
    \mathcal H
    =
    \{\langle a-a',u\rangle=0:\ a,a'\in A,\ a\ne a'\}.
\]
If $A$ is a singleton, the arrangement has the single full-dimensional cone $\R^d$.
Otherwise, $\mathcal H$ has $\size^{O(d)}$ full-dimensional cones, enumerable within the same
time bound.

On every full-dimensional open cone $C$, all slopes in $A$ are strictly ordered.
Thus each inner minimum selects a fixed slope and the outer maximum among these selected slopes
also selects a fixed slope.  Consequently, there is a slope $a_C\in A$ such that
\[
    \phi(u)=\langle a_C,u\rangle
    \qquad(u\in C).
\]
Define
\[
    S_e:=\{a_C:\ C\text{ is a full-dimensional cone of }\mathcal H\},
\]
with duplicates removed.  Every $a_C$ is attained as the gradient of $\phi$ on the open
full-dimensional cone $C$, so it is essentially active.  Conversely, if a slope is essentially active for
$\phi$ at the origin, then it is attained on some full-dimensional differentiability region of the local
PA fan.  The arrangement $\mathcal H$ refines that fan, so one of its full-dimensional cones lies
inside this region and contributes the same slope.  Hence $S_e$ is exactly the essentially active
slope set.

For a Lipschitz PA function, the Clarke subdifferential at a point is the convex hull of the limiting
gradients from neighboring full-dimensional differentiability regions.  Since $\phi$ is positively
homogeneous, all full-dimensional cones of its local fan are adjacent to the origin.  Therefore
\[
    \Csub f(w)=\Csub\phi(0)=\conv(S_e).
\]
The comparison $\dist(0,\Csub f(w))\le\eps$ is exactly
\[
    \min\left\{
        \left\|\sum_{a\in S_e}\lambda_a a\right\|_2^2:
        \lambda_a\ge0,\ \sum_{a\in S_e}\lambda_a=1
    \right\}
    \le \eps^2 .
\]
The set $S_e$ has size $\size^{O(d)}$, and Lemma~\ref{lem:exact-rational-qp} decides the displayed
comparison exactly in time polynomial in its encoded size.  Clarke-$\eps$-NO is obtained by
checking whether the optimum is greater than $\eps^2$.  The total running time is
$\size^{O(d)}$.
\end{proof}

\section{Deferred proofs for Theorem~\ref{thm:w1-hard}}\label{app:hardness-proofs}
This appendix gives the full reduction underlying Theorem~\ref{thm:w1-hard}.  We first build, from a \(k\)-Clique instance, a pair of polytopes \(X\subseteq Y\) such that the translate-containment condition \(s+Y\subseteq X\) distinguishes clique from no-clique instances.  Lemma~\ref{lem:frechet-translation} then turns this containment gadget into the Fr\'echet stationarity test.

\paragraph{Clique construction.}
Throughout this appendix, let \(G=([N],E)\) be an undirected graph with \(N\ge3\) and let \(k\ge2\).  Put
\[
    p_v:=(v,v^2)\in\Z^2,
    \qquad
    Q:=\conv\{p_v:v\in[N]\}\subseteq\R^2,
\]
and define the forbidden ordered pairs
\[
    \Forb:=\{(u,v)\in[N]^2:u=v\text{ or }\{u,v\}\notin E\}.
\]
Thus a tuple \((v_1,\ldots,v_k)\) is a \(k\)-clique exactly when no pair \((v_i,v_j)\), \(i<j\), is in \(\Forb\).  Let
\[
    \mathcal A_G:=\{(i,j,u,v):1\le i<j\le k,\ (u,v)\in\Forb\}.
\]
For \(z=(z_1,\ldots,z_k)\in(\R^2)^k\), define
\[
    g(z):=\sum_{r=1}^k\sigma_Q(z_r),
    \qquad
    \psi_{i,j,u,v}(z):=\langle p_u,z_i\rangle+\langle p_v,z_j\rangle+
    \sum_{r\in[k]\setminus\{i,j\}}\sigma_Q(z_r),
\]
and \(h(z):=\max_{\alpha\in\mathcal A_G}\psi_\alpha(z)\).  Finally set
\[
    Y:=\Csub g(0)=Q^k,
    \qquad
    X:=\Csub h(0).
\]

\paragraph{Geometric intuition.}
Before presenting the full proof, we first provide a geometric intuition to help understanding the high-level idea.
Every $p_v$ is an exposed vertex of $Q$, so the vertices of $Y=Q^k$ represent all ordered $k$-tuples. Fixing two tuple positions to a forbidden pair gives a product face of $Y$; $X$ is the convex hull of these faces and contains precisely the non-clique tuple vertices. A missing exposed clique vertex obstructs the containment required by Fr\'echet stationarity. The seesaw gadget in Lemma~\ref{lem:seesaw} turns that missing vertex into a neighboring zero-gradient region, giving Clarke stationarity.

\begin{lemma}[Clique polytope gadget]\label{lem:clique-polytope}
With the notation above, $X\subseteq Y=Q^k$.
Moreover, $X=Y$ if and only if $G$ has no $k$-clique.
If $G$ has a $k$-clique, then there is no $s\in\R^{2k}$ such that $s+Y\subseteq X$.
If $G$ has no $k$-clique, then the only $s\in\R^{2k}$ with $s+Y\subseteq X$ is $s=0$.
\end{lemma}

\begin{proof}
For \(\alpha=(i,j,u,v)\in\mathcal A_G\), let
\[
 F_\alpha:=Q\times\cdots\times Q\times\{p_u\}\times Q\times\cdots\times Q\times\{p_v\}\times Q\times\cdots\times Q\subseteq Q^k,
\]
where the singleton factors occur in blocks $i$ and $j$.
Then $\psi_\alpha=\sigma_{F_\alpha}$.
By the support-function calculus for polytopes~\citep{Ziegler1995}, the maximum of support functions is the support function of the convex hull of the union, and therefore
\[
 h=\max_\alpha \sigma_{F_\alpha}=\sigma_{\conv(\cup_\alpha F_\alpha)}.
\]
Therefore $X=\Csub h(0)=\conv(\cup_\alpha F_\alpha)\subseteq Q^k=Y$.
Because each $F_\alpha$ is a face of $Y$, the polytope $X$ is the convex hull of precisely those vertices of $Y$ that lie in at least one $F_\alpha$.
A vertex of $Y$ has the form $y(v_1,\dots,v_k):=(p_{v_1},\dots,p_{v_k})$.
It lies in $F_{i,j,u,v}$ if and only if $v_i=u$ and $v_j=v$.
Thus it is covered whenever some pair $(v_i,v_j)$ is forbidden.
Conversely, if no pair is forbidden, then the tuple corresponds to a $k$-clique and the vertex is not covered.
It also cannot lie in the convex hull of the covered vertices: by the exposed-face facts for polytopes~\citep{Ziegler1995}, every vertex of $Y=Q^k$ is exposed.
Indeed, for $p_v=(v,v^2)$ the vector $r_v:=(2v,-1)$ satisfies $\langle p_v,r_v\rangle-\langle p_a,r_v\rangle=(a-v)^2>0$ for $a\ne v$, and the product vector $(r_{v_1},\dots,r_{v_k})$ exposes $y(v_1,\dots,v_k)$ in $Y$.
Consequently, among the vertices of $Y$, the polytope $X$ contains exactly those corresponding to non-clique tuples.

If $G$ has no $k$-clique, every vertex of $Y$ is covered, so $Y=\conv(\vertices(Y))\subseteq X\subseteq Y$ and hence $X=Y$.
If $G$ has a $k$-clique $(v_1,\dots,v_k)$, the exposed vertex $y^*=y(v_1,\dots,v_k)$ is missing from $X$, so $X\subsetneq Y$.
Moreover, the exposing vector above gives a strict gap between $y^*$ and every covered vertex; hence $X$ is contained in a closed halfspace that cuts off a nonempty small cap of the full-dimensional polytope $Y$.
Thus $\vol(X)<\vol(Y)$, where volume is taken in $\R^{2k}$.
If some translate $s+Y$ were contained in $X$, then $\vol(Y)=\vol(s+Y)\le\vol(X)<\vol(Y)$, a contradiction.

Finally, in the no-clique case $X=Y$.
If $s+Y\subseteq Y$, then support functions give $\langle s,u\rangle+\sigma_Y(u)\le\sigma_Y(u)$ for all $u$, so $\langle s,u\rangle\le0$ for all $u$.
Applying this to both $u$ and $-u$ yields $s=0$.
\end{proof}

\begin{lemma}[Fr\'echet gadget]\label{lem:frechet-gadget}
For $f_F:=h-g$, the following equivalences hold:
\[
 G\text{ has no }k\text{-clique}\iff \Fsub f_F(0)=\{0\},
 \qquad
 G\text{ has a }k\text{-clique}\iff \Fsub f_F(0)=\emptyset.
\]
In particular, for every fixed $\eps\ge0$, $G$ has a $k$-clique if and only if $\dist(0,\Fsub f_F(0))>\eps$.
Moreover, $G$ has a $k$-clique if and only if the origin is not a local minimizer of $f_F$.
\end{lemma}

\begin{proof}
By Lemma~\ref{lem:frechet-translation}, $\Fsub f_F(0)=\{s:s+Y\subseteq X\}$.
The result is exactly Lemma~\ref{lem:clique-polytope}.
The final stationarity statement uses the convention $\dist(0,\emptyset)=+\infty$ and the fact that $\dist(0,\{0\})=0$.
The local-minimality statement follows from Lemma~\ref{lem:pa-local-model}, since $f_F$ is continuous PA and $0\in\Fsub f_F(0)$ is equivalent to local minimality at the origin.
\end{proof}

To obtain Clarke hardness, we use the seesaw gadget introduced by~\citep{TianSo2025}, where a scalar variable \(t\) and the branch \(-|t|/2\) are added.  In the no-clique case, \(f_F\equiv0\), so the gadget is just \(t/2\).  In the clique case, a negative direction of \(f_F\) makes the branch \(-|t|/2\) active on an open cone, where \(t/2-|t|/2=0\).  Thus, the zero slope appears in the clique case, and $0\in\Csub f_C(0,0)$.
\begin{lemma}[Seesaw gadget]\label{lem:seesaw}
Let $f_C(z,t)=t/2+\max\{f_F(z),-|t|/2\}$.
If $G$ has no $k$-clique, then $\dist(0,\Csub f_C(0,0))=1/2$.
If $G$ has a $k$-clique, then $0\in\Csub f_C(0,0)$.
\end{lemma}

\begin{proof}
If $G$ has no $k$-clique, then Lemma~\ref{lem:clique-polytope} gives $X=Y$.
Since $h=\sigma_X$ and $g=\sigma_Y$, we have $f_F\equiv0$.
Therefore $f_C(z,t)=t/2+\max\{0,-|t|/2\}=t/2$, and $\Csub f_C(0,0)=\{(0_{2k},1/2)\}$.
Thus $\dist(0,\Csub f_C(0,0))=1/2$.

Now suppose $G$ has a $k$-clique $(v_1,\dots,v_k)$.
For each block choose the exposing vector $r_i:=(2v_i,-1)\in\R^2$.
Then $p_{v_i}$ uniquely maximizes $\langle\cdot,r_i\rangle$ over $Q$, and for every $a\ne v_i$,
\[
 \langle p_{v_i}-p_a,r_i\rangle=(a-v_i)^2\ge1.
\]
Set $z_i:=T r_i$ and let $z=(z_1,\dots,z_k)$.
For a forbidden branch $\alpha=(i,j,u,v)$ of $f_F$, its value at $z$ is
\[
 B_\alpha(z)
 =\langle p_u,z_i\rangle-\sigma_Q(z_i)
   +\langle p_v,z_j\rangle-\sigma_Q(z_j)
 =T\langle p_u-p_{v_i},r_i\rangle
   +T\langle p_v-p_{v_j},r_j\rangle.
\]
Both terms are nonpositive, and at least one is at most $-T$, because $(v_i,v_j)$ is not forbidden whereas $(u,v)$ is forbidden.
Hence $B_\alpha(z)\le -T$ for every forbidden branch, and therefore $f_F(z)\le -T$.
Choose $T>1$.
Then $f_F(z)<-1$.
At the point $(z,1)$ we have $f_F(z)<-1<-1/2$.
By continuity, there is an open full-dimensional neighborhood $U$ of $(z,1)$ such that every $(\tilde z,\tilde t)\in U$ satisfies $\tilde t>0$ and $f_F(\tilde z)<-\tilde t/2$.
Let
\[
 C:=\{\lambda(\tilde z,\tilde t):\lambda>0,
      (\tilde z,\tilde t)\in U\}.
\]
Because $f_F$ is positively homogeneous, every $(z',t')\in C$ satisfies $t'>0$ and $f_F(z')<-t'/2=-|t'|/2$.
Thus, throughout the full-dimensional open cone $C$, $f_C(z',t')=t'/2-t'/2=0$.
Hence the affine piece with gradient $0$ is attained on a full-dimensional region adjacent to the origin.
For PA functions, the Clarke subdifferential at the origin is the convex hull of gradients on full-dimensional adjacent cells.
Therefore $0\in\Csub f_C(0,0)$.
\end{proof}

It remains to check that the hard functions just constructed are given efficiently in the two input models studied in the paper.
\begin{lemma}[Max--min and depth-two $n$-$\MC$ DC representations]\label{lem:representations}
The Fr\'echet gadget admits the explicit max--min representation
\[
 f_F(z)=
 \max_{1\le i<j\le k,\ (u,v)\in\Forb}
 \min_{a,b\in[N]}
 \bigl(\langle p_u-p_a,z_i\rangle+\langle p_v-p_b,z_j\rangle\bigr).
\]
It also admits an explicit polynomial-size DC representation $f_F=h-g$ in which $g$ is $1$-$\MC$ and $h$ is $2$-$\MC$.
The Clarke gadget $f_C(z,t)=t/2+\max\{f_F(z),-|t|/2\}$ admits both an explicit polynomial-size max--min representation and an explicit polynomial-size DC representation $f_C=H_C-G_C$ in which $G_C$ is $1$-$\MC$ and $H_C$ is $2$-$\MC$.
In particular, both gadgets have explicit $n$-$\MC$ DC representations of constant depth (indeed depth two after padding).
\end{lemma}

\begin{proof}
Since $\sigma_Q(z_i)=\max_{a\in[N]}\langle p_a,z_i\rangle$,
\[
 \langle p_u,z_i\rangle-\sigma_Q(z_i)=\min_{a\in[N]}\langle p_u-p_a,z_i\rangle,
 \qquad
 \langle p_v,z_j\rangle-\sigma_Q(z_j)=\min_{b\in[N]}\langle p_v-p_b,z_j\rangle.
\]
Substituting these identities into the branch formula for $f_F=h-g$ gives the displayed max--min representation.

The convex component $g(z)=\sum_{r=1}^k \sigma_Q(z_r)=\sum_{r=1}^k\max_{a\in[N]}\langle p_a,z_r\rangle$ is a $1$-$\MC$ formula.
Let $\mathcal A:=\{(i,j,u,v):1\le i<j\le k,\ (u,v)\in\Forb\}$.
For $\alpha=(i,j,u,v)\in\mathcal A$, define, for $r\in[k]$ and $a\in[N]$,
\[
 q_{\alpha,r,a}:=
 \begin{cases}
 p_u, & r=i,\\
 p_v, & r=j,\\
 p_a, & r\notin\{i,j\}.
 \end{cases}
\]
Then $\psi_\alpha(z)=\sum_{r=1}^k\max_{a\in[N]}\langle q_{\alpha,r,a},z_r\rangle$, because the max is degenerate in the two fixed blocks.
Thus
\[
 h(z)=\max_{\alpha\in\mathcal A}
      \sum_{r=1}^k\max_{a\in[N]}\langle q_{\alpha,r,a},z_r\rangle,
\]
which is a $2$-$\MC$ formula: an outer maximum over $\alpha$, followed by a sum over blocks and an inner maximum over $a$.
Consequently $f_F=h-g$ is a DC representation whose convex components have depths at most two after padding the $1$-$\MC$ formula for $g$.

For the Clarke gadget, the max--min form follows from $f_C(z,t)=\max\{f_F(z)+t/2,\min\{t,0\}\}$ and the max--min representation of $f_F$.
For the DC representation, define
\[
 G_C(z,t):=g(z)+\frac{|t|}{2}
 =\sum_{r=1}^k\max_{a\in[N]}\langle p_a,z_r\rangle
  +\max\left\{\frac{t}{2},-\frac{t}{2}\right\}.
\]
This is a $1$-$\MC$ formula.
Next set
\[
 H_C(z,t):=\max\left\{h(z)+\frac{|t|}{2},g(z)\right\}+\frac{t}{2}
 =\max\left\{h(z)+\max\{t,0\},\ g(z)+\frac{t}{2}\right\}.
\]
The first class of branches can be written as
\[
 \max_{\alpha\in\mathcal A,\ \eta\in\{0,1\}}
 \left[
 \sum_{r=1}^k\max_{a\in[N]}\langle q_{\alpha,r,a},z_r\rangle+\eta t
 \right],
\]
and the remaining branch is $\sum_{r=1}^k\max_{a\in[N]}\langle p_a,z_r\rangle+t/2$.
Therefore $H_C$ is an outer maximum over polynomially many branches, each branch being a sum of $k+1$ max-affine groups, where the last group is a singleton affine function in $t$.
Hence $H_C$ is $2$-$\MC$.
Finally,
\[
 H_C(z,t)-G_C(z,t)
 =\frac{t}{2}+\max\left\{h(z)-g(z),-\frac{|t|}{2}\right\}
 =f_C(z,t).
\]
All formulas have polynomial size in $N$ and $k$, and the depths are bounded by two.
\end{proof}

\paragraph{Proof of Theorem~\ref{thm:w1-hard}.}
We reduce from the \(k\)-Clique problem; its W[1]-hardness is classical~\citep{CyganEtAl2015}.
Lemma~\ref{lem:frechet-gadget} proves the Fr\'echet hardness equivalence and the corresponding non-local-minimality equivalence, and 
Lemma~\ref{lem:seesaw} proves the Clarke hardness equivalence for every fixed $\eps\in[0,1/2)$. Lemma~\ref{lem:representations} provides both required input representations for our constructed hard instances.
The Fr\'echet construction has dimension $2k$ and the Clarke construction has dimension $2k+1$.
In both constructions, the branches are indexed by at most $O(k^2N^2)$ forbidden choices together with polynomially many vertex choices, and all coefficients have polynomial bit length.
Thus the instances can be output in time $f(k)|(G,k)|^{O(1)}$ and the target dimension is bounded by a function of $k$, so both reductions are valid parameterized reductions.
Since Hypothesis~\ref{hyp:clique} rules out algorithms for $k$-Clique on $N$-vertex graphs running in time $\rho(k)N^{o(k)}$, and since the target encoding size satisfies $\size\le N^{O(1)}$ with $d=\Theta(k)$, the same reductions rule out algorithms for the Fr\'echet, non-local-minimality, and Clarke hard polarities running in time $\rho(d)\size^{o(d)}$.

\section{Deferred proofs for Section~\ref{sec:cnn}}\label{app:cnn-proofs}

This appendix transfers the clique gadgets from Appendix~\ref{app:hardness-proofs} to the displayed shallow CNN family.  First, we prove the local ReLU penalty that replaces the support-function terms used in the PA gadget.  Second, we show that the losses in our hard instances can be realized using a shared trainable \(1\times1\) convolution, ReLUs, fixed linear readouts, and max-pooling.  Finally, we verify the Fr\'echet and Clarke hardness reductions and provide an XP algorithm for the stationarity tests of the constructed shallow CNN loss family.

We use the same graph notation as in Appendix~\ref{app:hardness-proofs}.  Throughout, \(G=([N],E)\),
\[
    p_v=(v,v^2)\in\Z^2,
    \qquad
    \mathcal N_{\!1}(u)=\{a\in[N]: |a-u|=1\},
\]
and
\[
    \Forb=\{(u,v)\in[N]^2:u=v\text{ or }\{u,v\}\notin E\}.
\]
As before, \(\Forb\) records exactly the ordered pairs that cannot appear as two positions of a valid clique.  We assume \(N\ge2\); smaller instances are trivial or may be padded by isolated vertices.

\subsection{The local ReLU penalty}\label{app:cnn-penalty-proof}

\begin{lemma}[Separation property of the local ReLU penalty]\label{lem:cnn-penalty}
For \(u\in[N]\), define
\[
    P_u(s):=\sum_{a\in\mathcal N_{\!1}(u)}
    \relu\bigl(\langle p_a-p_u,s\rangle\bigr),
    \qquad s\in\R^2 .
\]
If \(p_u\) maximizes \(a\mapsto\langle p_a,s\rangle\) over \([N]\), then \(P_u(s)=0\).
Moreover, for \(r_c=(2c,-1)\), one has \(P_u(r_c)=0\) when \(u=c\), and \(P_u(r_c)\ge1\) when \(u\ne c\).
\end{lemma}

\begin{proof}
If \(p_u\) maximizes \(a\mapsto\langle p_a,s\rangle\) over \([N]\), then
\(\langle p_a-p_u,s\rangle\le0\) for every \(a\in\mathcal N_{\!1}(u)\), and hence every ReLU term in \(P_u(s)\) vanishes.

For the second claim, note that for \(r_c=(2c,-1)\),
\[
    \langle p_a,r_c\rangle=2ca-a^2=c^2-(a-c)^2.
\]
Thus \(a=c\) is the unique maximizer over \([N]\), so \(P_c(r_c)=0\).
If \(u<c\), then \(u+1\in\mathcal N_{\!1}(u)\) and
\[
    \langle p_{u+1}-p_u,r_c\rangle
    =(u-c)^2-(u+1-c)^2=2(c-u)-1\ge1.
\]
If \(u>c\), then \(u-1\in\mathcal N_{\!1}(u)\) and
\[
    \langle p_{u-1}-p_u,r_c\rangle
    =(u-c)^2-(u-1-c)^2=2(u-c)-1\ge1.
\]
In either case at least one ReLU term in \(P_u(r_c)\) is at least \(1\), proving \(P_u(r_c)\ge1\) whenever \(u\ne c\).
\end{proof}

\subsection{Realizing the losses by shallow ReLU/max-pooling CNNs}\label{app:cnn-layer-realization}

\paragraph{The network class.}
Following the shared-filter construction of \citet[Appendix~B.3]{TianSo2025}, fix one rational training tensor with $p$ channels, and let $Z$ be the matrix whose rows are the channel vectors at its spatial sites. A shared trainable no-bias $1\times1$ filter $\theta\in\R^p$ produces $h_\theta^{(0)}=Z\theta$. The subsequent layers satisfy $h_\theta^{(\ell)}=T_\ell(h_\theta^{(\ell-1)})$, where each $T_\ell$ is a fixed rational affine map, coordinatewise ReLU, or max-pooling over fixed windows. The depth is bounded by an absolute constant, and only $\theta$ is trained. The construction below realizes Equations~\eqref{eq:cnn-frechet-loss}--\eqref{eq:cnn-clarke-loss} within this class; their scalar outputs can be viewed as objectives with an identity loss head.

Let
\[
    \mathcal A_G:=
    \{(i,j,u,v):1\le i<j\le k,\ (u,v)\in\Forb\}.
\]
For the Fr\'echet-side construction, the trainable parameter is
\[
    \theta=z=(z_1,\ldots,z_k)\in(\R^2)^k=\R^{2k}.
\]
For every \(\alpha=(i,j,u,v)\in\mathcal A_G\) and \(a\in\mathcal N_{\!1}(u)\), create a spatial site with fixed channel vector
\[
    x_{\alpha,L,a}:=(0,\ldots,0,\underbrace{p_a-p_u}_{i\text{-th block}},0,\ldots,0)\in\Q^{2k}.
\]
For every \(b\in\mathcal N_{\!1}(v)\), create a spatial site with fixed channel vector
\[
    x_{\alpha,R,b}:=(0,\ldots,0,\underbrace{p_b-p_v}_{j\text{-th block}},0,\ldots,0)\in\Q^{2k}.
\]
Write
\[
    B_\alpha(z):=-P_u(z_i)-P_v(z_j).
\]
The trainable \(1\times1\) convolutional filter outputs
\(\langle p_a-p_u,z_i\rangle\) and \(\langle p_b-p_v,z_j\rangle\) on these sites.
After the ReLU layer, a fixed rational readout over the sites belonging to branch \(\alpha\) computes
\[
    -\sum_{a\in\mathcal N_{\!1}(u)}
      \relu\bigl(\langle p_a-p_u,z_i\rangle\bigr)
    -\sum_{b\in\mathcal N_{\!1}(v)}
      \relu\bigl(\langle p_b-p_v,z_j\rangle\bigr)
    =
    B_\alpha(z).
\]
This readout can be implemented as a fixed convolutional/readout window with all nonzero coefficients equal to \(-1\).
A global max-pooling layer over the branch channels computes
\[
    L_G^{\mathrm F}(z)=\max_{\alpha\in\mathcal A_G} B_\alpha(z).
\]
The number of sites is at most \(4|\mathcal A_G|=O(k^2N^2)\), and all coefficients have polynomial bit length.

For the Clarke-side construction, the trainable parameter is
\[
    \theta=(z,t)\in\R^{2k+1}.
\]
Use the same branch sites, padded by one zero coordinate, and add two scalar sites with channel vectors \(e_{2k+1}\) and \(-e_{2k+1}\).
After ReLU, these two sites output \(\relu(t)\) and \(\relu(-t)\).
A fixed rational readout computes
\[
    R_t(t):=-\frac12\relu(t)-\frac12\relu(-t)=-\frac{|t|}{2}.
\]
The final max-pooling layer computes \(\max\{L_G^{\mathrm F}(z),R_t(t)\}\), and the scalar output adds the fixed affine term \(t/2\).
Equivalently, this affine term can be written as \(\relu(t)/2-\relu(-t)/2\) using the same two scalar sites.
Thus the realized loss is exactly
\[
    L_G^{\mathrm C}(z,t)
    =
    \frac{t}{2}+\max\left\{L_G^{\mathrm F}(z),-\frac{|t|}{2}\right\}.
\]

\paragraph{A standard supervised loss head.}
Let $S_G$ denote either displayed scalar output $L_G^{\mathrm F}$ or $L_G^{\mathrm C}$. Assign the fixed training example label $y=-1$ and use the standard hinge loss $\ell(y,s)=\max\{0,1-ys\}$, giving
\[
    \widehat L_G(\theta)=\max\{0,1+S_G(\theta)\}.
\]
The stationarity query point is $\theta=0$, where every ReLU penalty vanishes and $S_G(0)=0$. Continuity gives $S_G(\theta)>-1$ near the origin, so $\widehat L_G=1+S_G$ throughout that neighborhood. Thus the Fr\'echet and Clarke subdifferentials and local-minimality behavior at the query point agree. The fixed affine/ReLU head adds no trainable parameters and preserves constant depth. The lower bounds in Theorem~\ref{thm:cnn-hardness} therefore also hold for these one-example hinge losses, and for the constant-depth loss class containing them.

\subsection{Proof of Theorem~\ref{thm:cnn-hardness}}\label{app:cnn-hardness-proof}

\begin{proof}[Proof of Theorem~\ref{thm:cnn-hardness}]
The functions \(L_G^{\mathrm F}\) and \(L_G^{\mathrm C}\) are continuous, PA, and positively homogeneous.
Moreover \(L_G^{\mathrm F}(z)\le0\) for all \(z\), because every penalty \(P_u\) is nonnegative.

\paragraph{Fr\'echet reduction: no-clique case.}
Assume that \(G\) has no \(k\)-clique.
Fix arbitrary \(z=(z_1,\ldots,z_k)\).
For each block \(r\in[k]\), choose
\[
    v_r\in\argmax_{a\in[N]}\langle p_a,z_r\rangle .
\]
Since \((v_1,\ldots,v_k)\) is not a clique, there exist \(i<j\) such that \((v_i,v_j)\in\Forb\).
For the branch \(\alpha=(i,j,v_i,v_j)\), Lemma~\ref{lem:cnn-penalty} gives
\[
    P_{v_i}(z_i)=P_{v_j}(z_j)=0,
\]
and hence \(B_\alpha(z)=0\).
All branch scores are nonpositive, so \(L_G^{\mathrm F}(z)=0\).
As \(z\) was arbitrary, \(L_G^{\mathrm F}\equiv0\).
Therefore
\[
    \Fsub L_G^{\mathrm F}(0)=\{0\},
    \qquad
    \dist(0,\Fsub L_G^{\mathrm F}(0))=0.
\]

\paragraph{Fr\'echet reduction: clique case.}
Assume that \(G\) has a \(k\)-clique \((c_1,\ldots,c_k)\).
For \(T>0\), set
\[
    z_i:=T(2c_i,-1).
\]
For any forbidden branch \(\alpha=(i,j,u,v)\), the ordered pair \((c_i,c_j)\) is allowed whereas \((u,v)\) is forbidden, so \(u\ne c_i\) or \(v\ne c_j\).
Lemma~\ref{lem:cnn-penalty} gives \(P_u(z_i)\ge T\) in the first case and \(P_v(z_j)\ge T\) in the second.
Thus
\[
    B_\alpha(z)\le -T
\]
for every branch, and \(L_G^{\mathrm F}(z)\le -T<0\).

We now show that \(\Fsub L_G^{\mathrm F}(0)=\emptyset\).
Suppose \(g\in\Fsub L_G^{\mathrm F}(0)\).
Since \(L_G^{\mathrm F}\) is positively homogeneous, the Fr\'echet subgradient inequality gives
\[
    \langle g,d\rangle\le L_G^{\mathrm F}(d)\qquad\forall d.
\]
Taking \(d=g\) and using \(L_G^{\mathrm F}\le0\) yields \(\|g\|^2\le0\), so \(g=0\).
But \(0\in\Fsub L_G^{\mathrm F}(0)\) would imply \(L_G^{\mathrm F}(d)\ge0\) for all \(d\), contradicting the negative clique direction above.
Hence
\[
    \Fsub L_G^{\mathrm F}(0)=\emptyset.
\]
By the convention \(\dist(0,\emptyset)=+\infty\), clique instances are exactly the instances with
\[
    \dist(0,\Fsub L_G^{\mathrm F}(0))>\eps
\]
for every fixed \(\eps\ge0\).
The same two cases also decide local minimality: when there is no clique, \(L_G^{\mathrm F}\equiv0\), so the origin is a local minimizer; when there is a clique, the negative direction above shows that the origin is not a local minimizer.

\paragraph{Clarke reduction: no-clique case.}
If \(G\) has no \(k\)-clique, then \(L_G^{\mathrm F}\equiv0\).
Thus
\[
    L_G^{\mathrm C}(z,t)
    =
    \frac{t}{2}+\max\left\{0,-\frac{|t|}{2}\right\}
    =
    \frac{t}{2}.
\]
Consequently
\[
    \Csub L_G^{\mathrm C}(0,0)=\{(0,\ldots,0,1/2)\},
    \qquad
    \dist(0,\Csub L_G^{\mathrm C}(0,0))=\frac12.
\]
For every fixed \(\eps\in[0,1/2)\), no-clique instances are therefore Clarke-\(\eps\)-NO instances.

\paragraph{Clarke reduction: clique case.}
If \(G\) has a \(k\)-clique, the Fr\'echet part gives \(z^\star\) with
\[
    L_G^{\mathrm F}(z^\star)<-1.
\]
At \((z^\star,1)\), the branch \(-|t|/2\) has value \(-1/2\) and is strictly larger than \(L_G^{\mathrm F}(z^\star)\).
By continuity, there is a full-dimensional open neighborhood \(U\) of \((z^\star,1)\) on which \(t>0\) and the branch \(-|t|/2\) is strictly selected.
On \(U\),
\[
    L_G^{\mathrm C}(z,t)=\frac{t}{2}-\frac{t}{2}=0.
\]
Since \(L_G^{\mathrm C}\) is positively homogeneous, the cone generated by a sufficiently small open subset of \(U\) is a full-dimensional region whose closure contains the origin, and on this region the affine gradient is \(0\).
For PA functions, the Clarke subdifferential at the origin is the convex hull of gradients attained on adjacent full-dimensional regions.
Thus
\[
    0\in\Csub L_G^{\mathrm C}(0,0).
\]
This proves that clique instances are exactly the instances with
\[
    \dist(0,\Csub L_G^{\mathrm C}(0,0))\le\eps
\]
for every fixed \(\eps\in[0,1/2)\).

\paragraph{Parameterized and conditional running-time consequences.}
The Fr\'echet construction has trainable dimension \(p=2k\); the Clarke construction has trainable dimension \(p=2k+1\).
The number of sites, fixed readout coefficients, and pooling windows is polynomial in \(N\) and \(k\): there are \(O(kN)\) local penalty channels and \(O(k^2N^2)\) forbidden-branch channels, all with rational coefficients of polynomial bit length.
Thus the CNN instances can be output in time \(f(k)|(G,k)|^{O(1)}\), with \(p\) bounded by a function of \(k\), so the reductions are valid parameterized reductions.
If any of the hard CNN problems above were solvable in time \(\rho(p)M^{o(p)}\), then \(k\)-Clique on \(N\) vertices would be solvable in time \(\rho'(k)N^{o(k)}\) for a computable function \(\rho'\), contradicting Hypothesis~\ref{hyp:clique}.
\end{proof}

\subsection{Conversion to explicit max--min form}\label{app:cnn-maxmin-xp-proof}

\begin{proof}[Proof of Corollary~\ref{cor:cnn-complete}]
For \(\alpha=(i,j,u,v)\in\mathcal A_G\), write
\[
    \ell^{i,u}_a(z):=\langle p_a-p_u,z_i\rangle
    \quad(a\in\mathcal N_{\!1}(u)),
    \qquad
    \ell^{j,v}_b(z):=\langle p_b-p_v,z_j\rangle
    \quad(b\in\mathcal N_{\!1}(v)).
\]
The identity
\[
    -\relu(\ell)=\min\{0,-\ell\}
    =
    \min_{\omega\in\{0,1\}}(-\omega\ell)
\]
and distributivity of independent minima over sums give
\[
\begin{aligned}
    B_\alpha(z)
    &=
    \min_{\omega\in\{0,1\}^{\mathcal N_{\!1}(u)},\,
          \tau\in\{0,1\}^{\mathcal N_{\!1}(v)}}
       \left(
       -\sum_{a\in\mathcal N_{\!1}(u)}\omega_a\ell^{i,u}_a(z)
       -\sum_{b\in\mathcal N_{\!1}(v)}\tau_b\ell^{j,v}_b(z)
       \right).
\end{aligned}
\]
Since
\[
    |\mathcal N_{\!1}(u)|,\ |\mathcal N_{\!1}(v)|\le2,
\]
each branch has at most \(2^4=16\) affine pieces.
Thus
\[
    L_G^{\mathrm F}(z)
    =
    \max_{\alpha\in\mathcal A_G}
    \min_{\omega,\tau}
       \left(
       -\sum_{a\in\mathcal N_{\!1}(u)}\omega_a\ell^{i,u}_a(z)
       -\sum_{b\in\mathcal N_{\!1}(v)}\tau_b\ell^{j,v}_b(z)
       \right)
\]
is an explicit max--min formula of size \(O(|\mathcal A_G|)\) up to a constant factor, and hence of size polynomial in \(M\).

For the Clarke loss, use
\[
    \frac{t}{2}-\frac{|t|}{2}=\min\{t,0\}.
\]
Combining this identity with the preceding branch expansion yields
\[
    L_G^{\mathrm C}(z,t)
    =
    \max\left\{
    \max_{\alpha\in\mathcal A_G}\min_{\omega,\tau}
    \left[
       \frac{t}{2}
       -\sum_{a\in\mathcal N_{\!1}(u)}\omega_a\ell^{i,u}_a(z)
       -\sum_{b\in\mathcal N_{\!1}(v)}\tau_b\ell^{j,v}_b(z)
    \right],\;
    \min\{t,0\}
    \right\}.
\]
This is again an explicit max--min formula of size polynomial in \(M\) and dimension \(p=2k+1\).
Applying Theorem~\ref{thm:fixed-d-xp} gives deterministic time \(M^{O(p)}\) for Fr\'echet-\(\eps\)-YES and Clarke-\(\eps\)-NO, and complementing the deterministic output gives the same upper bound for the opposite polarities.
Together with Theorem~\ref{thm:cnn-hardness}, this proves Corollary~\ref{cor:cnn-complete}.
\end{proof}

\paragraph{Scope of the upper bound and network parameters.}
The $M^{O(p)}$ upper bound in Corollary~\ref{cor:cnn-complete} concerns the displayed constructed family. It does not assert an XP algorithm for the entire network class described above. In the construction, the channel dimension equals $p$, so the bound is XP in channel dimension and formally in the larger total tensor dimension. The output dimension is always one; an XP algorithm parameterized only by output dimension would  imply $P=NP$.

\end{document}